\documentclass[10pt]{amsart}

\usepackage{hyperref}
\usepackage{amsmath}
\usepackage{graphicx}
\usepackage{tikz}
\usepackage{tikz-cd}
\usepackage{tikz-3dplot}
\usepackage{subcaption}

\usetikzlibrary{snakes}
\usepackage[all]{xy}
\usepackage{framed}
\usepackage{mathrsfs}%For mathscr
\usepackage{amssymb}
\usepackage{stmaryrd}

\usepackage{mathabx}

\usepackage{colonequals} %\colonequal

\theoremstyle{plain}
\newtheorem{theorem}{Theorem}[section]
\newtheorem{lemma}[theorem]{Lemma}
\newtheorem{proposition}[theorem]{Proposition}
\newtheorem{corollary}[theorem]{Corollary}

\theoremstyle{definition}
\newtheorem{definition}[theorem]{Definition}
\newtheorem{example}[theorem]{Example}

\numberwithin{equation}{section}

\newcommand{\QQ}{\mathbb{Q}}
\newcommand{\RR}{\mathbb{R}}
\newcommand{\ZZ}{\mathbb{Z}}

\newcommand{\mc}{\mathcal}

\def\IH{\mathit{IH}}

\DeclareMathOperator{\rank}{rank}

\DeclareMathOperator{\Poin}{\mathsf{Poin}}
\DeclareMathOperator{\relint}{relint}

\newcommand{\leqmrot}{\mathbin{\rotatebox[origin=c]{-40}{$\leq$}}}
\newcommand{\leqprot}{\mathbin{\rotatebox[origin=c]{50}{$\leq$}}}

\begin{document}

\title[Conic decomposition]{Conic decomposition of a toric variety and its application to cohomology}

\author{Seonjeong Park}
\address{Department of Mathematics Education, Jeonju University,  303, Cheonjam-ro, Wansan-gu, Jeonju-si, Jeollabuk-do, 55069, Republic of Korea}
\email{seonjeongpark@jj.ac.kr}

\author{Jongbaek Song}
\address{School of Mathematics, KIAS, 
85 Hoegiro Dongdaemun-gu, Seoul 02455, Republic of Korea}
\email{jongbaek@kias.re.kr}

%\date{\today}

%\thanks{}

\maketitle 
\abstract 
We introduce the notion of a \emph{conic sequence} of a convex polytope. It is a way of building up a polytope starting from a vertex and attaching faces one by one with certain regulations. We apply this to a toric variety to obtain an iterated cofibration structure on it. This allows us to prove several vanishing results in the rational cohomology of a toric variety and to calculate Poincar\'e polynomials for a large class of singular toric varieties. 
\endabstract

\section{Introduction}
Let $P$ be a convex polytope of dimension $n$ and $f(P)\colonequals (f_0, \dots, f_n)$ its face vector, i.e., $f_i$ is the number of $i$-dimensional faces of $P$. In literature, the vector $f(P)$ is often called the \emph{$f$-vector} or the \emph{face vector} of $P$. If $P$ is a simple polytope, one can associate the \emph{$h$-vector} $h(P)=(h_0, \dots, h_n)$ defined by the following equation
\begin{equation}\label{eq_fvec_hvec}
\sum_{k=0}^n f_k (x-1)^k = \sum_{k=0}^n h_k x^k.
\end{equation}

From the geometric point of view, if $P$ is a full dimensional lattice polytope in $M\otimes_\ZZ \RR$ for some $n$-dimensional lattice $M$, one can associate a projective toric variety $X_P$ of complex dimension $n$ whose orbit space with respect to the action of the $n$-dimensional compact torus $T^n$ is combinatorially equivalent to $P$. In particular, if $P$ is a simple lattice polytope, the corresponding toric variety $X_P$ is a rationally smooth projective toric variety, which could be  understood as a \emph{toric orbifold} \cite{DJ,PS, BSS}. In this case, the singular cohomology of $X_P$ with rational coefficients vanishes in odd degrees. Moreover, plugging $t^2$ into $x$ in  the right-hand side of \eqref{eq_fvec_hvec}, we get the Poincar\'e polynomial $\Poin(X_P,t)$ of $X_P$, namely $h_k$ is the rank of the $2k$-th cohomology group. In other words, the Betti numbers of a rationally smooth toric variety are completely determined by its underlying combinatorics. This is one of the key observations in Stanley's proof of the $g$-conjecture, we refer to  \cite{McM, Sta}. 

If we turn our attention to more general lattice polytopes other than simple polytopes, then the associated projective toric varieties are no longer rationally smooth. In general, for a singular projective variety, we often consider the intersection cohomology $\IH^\ast(-)$ because it carries geometric and topological properties of singular strata. Moreover, there are certain similarities between  $\IH^\ast(-)$ and the singular cohomology of a smooth variety, for instance 
Poincar\'e duality or the hard Lefschetz property. Especially for a projective toric variety $X_P$, the Betti numbers of $\IH^\ast(X_P)$ are combinatorial invariants. Indeed, Stanley \cite{Sta2} defines the \emph{toric $h$-polynomial} of $P$, which agrees with \eqref{eq_fvec_hvec} when $P$ is a simple polytope, and he shows that the coefficients in the toric $h$-polynomial give the Betti numbers of $\IH^\ast(X_P)$. 

However, if we stick to the singular cohomology of a toric variety with singularities beyond orbifold singularities, it is somehow more subtle to deal with than the intersection cohomology because it is far from the Poincar\'e duality, as well as its rational cohomology is not concentrated in even degrees in general. For instance, the $6$-dimensional toric variety associated with the $3$-dimensional cross polytope, that is the convex hull of $(\pm 1,0,0), (0,\pm 1, 0)$ and $(0,0,\pm 1)$ in $\ZZ^3 \otimes_\ZZ \RR$, has a nontrivial rational cohomology class in degree~$3$. We refer to \cite[Example 12.3.8]{CLS} whose computation is based on the results in \cite{Fis, Jor}. Moreover, Barthel, et. al. \cite[Example 3.5]{BBFK} provide two $6$-dimensional toric varieties $X$ and $X'$ associated with combinatorially equivalent polytopes such that $\rank H^3(X;\QQ)\neq \rank H^3(X';\QQ)$. This tells us that the Betti numbers are neither  combinatorial invariants in general, nor obtainable by the $h$-vector defined in \eqref{eq_fvec_hvec}.

The aim of this paper is to explore the category of toric varieties other than toric orbifolds whose Betti numbers for the singular cohomology can be calculated by the underlying combinatorics. Our main idea is to consider a topological stratification 
\begin{equation}\label{eq_stratification}
    \{ \text{a point}\} = X_1 \subset X_2 \subset \cdots \subset X_P
\end{equation}
on $X_P$ such that each stratum is a mapping cone, see \eqref{eq_conic_on_X}. We call such a stratification a \emph{conic decomposition} of $X_P$, which is named after the notion of a conic decomposition of a map $f\colon X \to Y$ introduced by Jeffrey A. Strom.
In this case, we consider the map $\{\text{a point}\} \hookrightarrow X_P$, regarding $X_P$ as a pointed space.  A conic decomposition of $X_P$ can be parametrized by a sequence of certain subspaces in $P$, which we call a \emph{conic sequence} of $P$. Beginning with $P$ itself as the initial term, the sequence continues inductively by deleting faces whose geometric realization agrees with a cone on a polytope if it has, and the sequence stops otherwise. We discuss explicit combinatorial and topological definitions of a conic sequence of $P$ in Section \ref{sec_conic_polytope}. 

We emphasize again that each stratum $X_i$ in \eqref{eq_stratification} is a mapping cone of a map $\Phi_i \colon Y_i \to X_{i-1}$ for some topological space $Y_i$. Hence, the topology of $X_i$ is dominated by the topology of $Y_i$ together with the attaching map $\Phi_i$. Accordingly, if a toric variety $X_P$ admits a conic decomposition, then one can study the topology of $X_P$ by analyzing each cofiber sequence 
\begin{equation}\label{eq_cofiber_seq}
Y_i \xrightarrow{\Phi_i} X_{i-1} \to X_i
\end{equation}
successively. 

In Section \ref{sec_conic_decomp_X_P}, we focus on the case where the conic sequence of $P$ is given by iterated deletions of the cones of simple polytopes, where we denote these simple polytopes by $C_i$'s  throughout this paper. This is equivalent to say that $Y_i$ in \eqref{eq_cofiber_seq} is an orbifold $S^1$-fibration over a toric orbifold corresponding to $C_i$ (Lemma \ref{lem_Y_i_for_simple_poly}). A particular example of this setup is the case where each $C_i$ is given by a simplex, which has been  introduced in \cite{BSS} and further studied in subsequent papers \cite{BSS2, BNSS, SaSo2, LMPS}, where the conic sequence of this particular type is referred as a \emph{retraction sequence}.

In Section \ref{sec_cohomology}, we use a conic decomposition of $X_P$ to calculate the cohomology group of $X_P$ with rational coefficients. We first see how the cohomology group of $Y_i$ affects the cohomology of $X_i$. Then we prove several vanishing theorems for the cases discussed in Section~\ref{sec_conic_decomp_X_P}. In particular, we verify in Theorem \ref{thm_poincare_poly} that a certain class of toric varieties including all toric orbifolds enjoys \eqref{eq_fvec_hvec} to obtain their Poincar\'e polynomials.

\subsection*{Acknowledgements}
Some part of this paper is motivated by Jeffrey A. Strom's talk at Fields institute during the Thematic program on Toric Topology and Polyhedral Products in 2020. The second named author is grateful to the organizers of the program and to Jeffrey A. Strom for his nice presentation and communication about the usage of some terminologies. 

Park has been supported by Basic Science Research Program through the National Research Foundation of Korea (NRF) funded by the Ministry of Science and ICT (NRF-2020R1A2C1A01011045).
Song has been supported by Basic Science Research Program through the National Research Foundation of Korea (NRF) funded by the Ministry of Education (NRF-2018R1D1A1B07048480) and a KIAS Individual Grant (MG076101) at Korea Institute for Advanced Study. This research was supported by the Research Grant of Jeonju University in 2021.

\section{Conic sequence of a convex polytope}\label{sec_conic_polytope}
In this section, we introduce the notion of a conic sequence of a convex polytope.
A polytopal complex $\mc{P}$ is a finite collection of polytopes in an Euclidean space such that
\begin{enumerate}
    \item[(i)] the empty polytope is in $\mc{P}$,
    \item[(ii)] if $E \in\mc{P}$, then all the faces of $E$ are also in $\mc{P}$, and
    \item[(iii)] the intersection $E\cap F$ of two polytopes $E, F\in \mc{P}$ is a face both of $E$ and of $F$.
\end{enumerate}
The dimension $\dim \mc{P}$ is the largest dimension of a polytope in $\mc{P}$. The geometric realization of $\mc{P}$ is the point set $|\mc{P}| \colonequals \bigcup_{E\in\mc{P}} E$.
The combinatorial structure of a polytopal complex $\mc{P}$ is captured by its face poset $L(\mc{P}) := (\mc{P},\subseteq)$, that is, the finite set of polytopes in $\mc{P}$, ordered by inclusion. We define two polytopal complexes to be combinatorially equivalent if their face posets are isomorphic as posets. We refer to \cite[Chapter 5]{Zie} for more details.
In what follows, we assume that $\mc{P}$ is of dimension $n$ otherwise stated and we denote by $\ell$ the number of all the vertices of $\mc{P}$.

We now define a sequence $\{\mathcal{P}_i\}_{1\leq i \leq \ell}$ of polytopal subcomplexes $\mathcal{P}_i$ of $\mc{P}$ inductively as follows.  For technical reasons, we construct the sequence with decreasing indices.  We set the initial term $\mc{P}_\ell$ by $\mc{P}$. Given $\mc{P}_{i}$ for $i \leq \ell$, we take a vertex $v_{i}$ of $\mc{P}_{i}$ such that there exists a unique maximal element $E_i$ in the subposet $\{F \in L(\mc{P}_{i}) \mid v_i \subseteq F\}$. We call this vertex $v_i$ a \emph{cone vertex} of $\mc{P}_i$. Then we define the next term $\mc{P}_{i-1}$ by
\begin{equation}\label{eq_conic_seq_next_term}
\mc{P}_{i-1}\colonequals \mc{P}_i - [v_i,E_i],
\end{equation}
where $[v_i, E_i]\colonequals \{F\in \mc{P}_i \mid v_i \subseteq F \subseteq E_i\}$. We call the sequence~$\{\mathcal{P}_i\}_{1\leq i \leq \ell}$ defined as above a \emph{conic sequence} of $\mc{P}$ if 
the sequence ends up with $\mathcal{P}_1=\{\emptyset,v_1\}$ for some vertex $v_1$ of $\mc{P}$. If in particular $\mc{P}$ is the complex associated with a polytope $P$, namely the complex of all faces of $P$, then we simply call the sequence $\{\mathcal{P}_i\}_{1\leq i \leq \ell}$ a conic sequence of a polytope $P$. 

\begin{example} \label{ex_square}
Consider the square $P$ as a $2$-dimensional polytope. Its associated polytopal complex $\mc{P}$ is given by $\{\emptyset,p_1, p_2, p_3, p_4, e_{14}, e_{12}, e_{23}, e_{34}, P\}$, where $p_1, \dots, p_4$ are vertices of $P$ and each $e_{ij}$ is the edge connecting $p_i$ and $p_j$. Recall that the initial term $\mc{P}_4$ of a conic sequence is $\mc{P}$ itself. Now,
we may take $p_1$ as a cone vertex of $\mc{P}_4$ because the subposet $\{F \in {L}(\mc{P}) \mid p_1 \subseteq F\}$ has the unique maximal element $P$. Hence, we have 
\[
\mc{P}_3=\mc{P}-\{p_1,e_{12},e_{14}, P\} = \{\emptyset,p_2, p_3, p_4, e_{23}, e_{34}\}.
\] 
Since the subposet $\{F \in {L}(\mc{P}_3) \mid p_3 \subseteq F\}$ has two maximal elements $e_{23}$ and $e_{34}$, the vertex $p_3$ does not satisfy the condition for a cone vertex. Instead, one can take either $p_2$ or $p_4$ as a cone vertex to construct the next term of a conic sequence. Figure~\ref{fig_square} shows an example of a conic sequence of the polytope $P$, where the cone vertices are highlighted. 
\end{example}

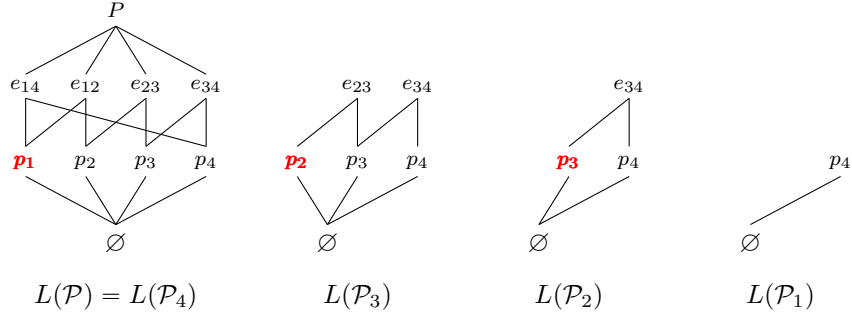
\begin{figure}
\begin{tikzpicture}[scale=0.8]

\begin{scope}%[xshift=100]

\node[above] at (2.5, 2) {\footnotesize$P$};

\draw (2.5,2)--(1,1.2);
\draw (2.5,2)--(2,1.2);
\draw (2.5,2)--(3,1.2);
\draw (2.5,2)--(4,1.2);

\node at (1, 1) {\footnotesize$e_{14}$};
\node at (2, 1) {\footnotesize$e_{12}$};
\node at (3, 1) {\footnotesize$e_{23}$};
\node at (4, 1) {\footnotesize$e_{34}$};

\draw (1,0.8)--(1,0);
\draw (1,0.8)--(4,0);

\draw (2,0.8)--(1,0);
\draw (2,0.8)--(2,0);

\draw (3,0.8)--(2,0);
\draw (3,0.8)--(3,0);

\draw (4,0.8)--(3,0);
\draw (4,0.8)--(4,0);

\node[below, red] at (1, 0) {\footnotesize$\pmb{p_1}$};
\node[below] at (2, 0) {\footnotesize$p_2$};
\node[below] at (3, 0) {\footnotesize$p_3$};
\node[below] at (4, 0) {\footnotesize$p_4$};

\draw (1,-0.5)--(2.5,-1.3);
\draw (2,-0.5)--(2.5,-1.3);
\draw (3,-0.5)--(2.5,-1.3);
\draw (4,-0.5)--(2.5,-1.3);
\node[below] at (2.5,-1.3) {\footnotesize$\emptyset$};

\node at (2.5, -2.5) {${L}(\mathcal{P})={L}(\mathcal{P}_4)$};

\begin{scope}[xshift=100]
\node at (3, 1) {\footnotesize$e_{23}$};
\node at (4, 1) {\footnotesize$e_{34}$};

\draw (3,0.8)--(2,0);
\draw (3,0.8)--(3,0);

\draw (4,0.8)--(3,0);
\draw (4,0.8)--(4,0);

\node[below, red] at (2, 0) {\footnotesize$\pmb{p_2}$};
\node[below] at (3, 0) {\footnotesize$p_3$};
\node[below] at (4, 0) {\footnotesize$p_4$};

\draw (2,-0.5)--(2.5,-1.3);
\draw (3,-0.5)--(2.5,-1.3);
\draw (4,-0.5)--(2.5,-1.3);
\node[below] at (2.5,-1.3) {\footnotesize$\emptyset$};
\node at (3, -2.5) {${L}(\mathcal{P}_3)$};
\end{scope}

\begin{scope}[xshift=200]
\node at (4, 1) {\footnotesize$e_{34}$};

\draw (4,0.8)--(3,0);
\draw (4,0.8)--(4,0);

\node[below,red] at (3, 0) {\footnotesize$\pmb{p_3}$};
\node[below] at (4, 0) {\footnotesize$p_4$};

\draw (3,-0.5)--(2.5,-1.3);
\draw (4,-0.5)--(2.5,-1.3);
\node[below] at (2.5,-1.3) {\footnotesize$\emptyset$};

\node at (3, -2.5) {${L}(\mathcal{P}_2)$};
\end{scope}

\begin{scope}[xshift=300]
\node[below] at (4, 0) {\footnotesize$p_4$};
\draw (4,-0.5)--(2.5,-1.3);
\node[below] at (2.5,-1.3) {\footnotesize$\emptyset$};
\node at (3, -2.5) {${L}(\mathcal{P}_1)$};
\end{scope}
\end{scope}

\end{tikzpicture}
\caption{The face posets associated with a conic sequence of a square.}
\label{fig_square}
\end{figure}

Note that not every convex polytope admits a conic sequence. For example, the octahedron does not admit a conic sequence. In fact, any $n$-gonal bipyramid does not admit a conic sequence for $n\geq 4$.

\begin{definition}\label{def_conic}
A convex polytope $P$ is said to be \emph{conic} if it admits a conic sequence. 
\end{definition}

Since each $\mathcal{P}_i$ is a polytopal subcomplex of $\mathcal{P}_{i+1}$,  a conic sequence $\{\mathcal{P}_i\}_{1\leq i \leq \ell}$ of a polytopal complex $\mc{P}$ associated with a conic polytope $P$ induces an increasing filtration 
\begin{equation}\label{eq_filt_P}
P_1 \hookrightarrow \cdots  \hookrightarrow P_i \hookrightarrow P_{i+1} \hookrightarrow  \cdots \hookrightarrow P_{\ell-1}\hookrightarrow P_\ell=P
\end{equation}
of $P$, where each $P_i$ is the geometric realization $|\mc{P}_i|$. In particular, $P_i-P_{i-1}$ agrees with the union 
\[
\bigcup_{F\in [v_i, E_i]} \relint (F),
\]
where $\relint(F)$ means the relative interior of the face $F$ in $P$. 

Let $C_i$ be the intersection of $E_i$ with an affine hyperplane separating $v_i$ and $P_{i-1}\cap E_i$. Then
we can characterize $P_i$ topologically by the mapping cone of the map
\begin{equation}\label{eq_attaching_map}
\phi_i \colon C_i \to P_{i-1}
\end{equation}
whose image $\phi(x)$ of $x\in C_i$  is defined by the intersection of $P_{i-1}$ and the straight line passing through $x$ and $v_i$. See Figure \ref{fig_attaching_map} for a pictorial description of $\phi_i$ for the polytope given in Example \ref{ex_square}, where we illustrate $\phi_4$ with respect to the first cone vertex~$v_4$. We finally notice that since there is only one vertex of $E_i$ not contained in $P_{i-1}$, the dimension of $E_i$ is at most $\dim P_{i-1}+1$, so we have $\dim C_i \leq \dim P_{i-1}$. 

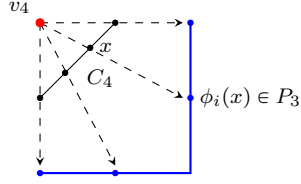
\begin{figure}
\begin{tikzpicture}
\draw[blue,thick] (2,2)--(2,0)--(0,0);

\filldraw[blue] (0,0) circle (1pt);
\filldraw[blue] (2,1) circle (1pt);
\filldraw[blue] (1,0) circle (1pt);
\filldraw[blue] (2,2) circle (1pt);

\draw (0,1)--(1,2);
\filldraw (0,1) circle (1pt);
\filldraw (1,2) circle (1pt);
\filldraw (2/3,5/3) circle (1pt);
\filldraw (1/3,4/3) circle (1pt);

\node[right] at (2/3,5/3) {\footnotesize$x$};
\node[right] at (2,1) {\footnotesize$\phi_i(x)\in P_3$};

\node[below right] at (0.5, 1.5) {\footnotesize$C_4$};
\draw[dashed, -stealth] (0,2)--(1.9,2);
\draw[dashed, -stealth] (0,2)--(1.9,1);
\draw[dashed, -stealth] (0,2)--(1,0.1);
\draw[dashed, -stealth] (0,2)--(0,0.1);

\filldraw[red] (0,2) circle (1.5pt);
\node[above left] at (0,2) {\footnotesize$v_4$};

\end{tikzpicture}
\caption{A pictorial description of $\phi_i$.}
\label{fig_attaching_map}
\end{figure}

In order to keep the information of $C_i$'s for $2\leq i\leq \ell$, we extend the filtration in~\eqref{eq_filt_P} to the following diagram
\begin{equation}\label{eq_conic_decomp}
\begin{tikzcd}[column sep=small]
C_2 \dar{\phi_2} &  & C_{i} \dar{\phi_i}   &C_{i+1} \dar{\phi_{i+1}}  &   & C_{\ell} \dar{\phi_{\ell}}  &  \\
P_1 \rar[hook] & \cdots   \rar[hook]  &  P_{i-1}  \rar[hook]   & P_{i} \rar[hook] & \cdots \rar[hook]  & P_{\ell-1} \rar[hook]  & P_{\ell}=P, 
\end{tikzcd}
\end{equation}
and we summarize the information of \eqref{eq_conic_decomp} as follows:
\begin{enumerate}
\item $P_1$ is a vertex of $P$;
\item each $C_i$ for $2\leq i \leq \ell$ is a convex polytope with $\dim C_i \leq \dim P_{i-1}$; and
\item each $P_{i}$ for $2 \leq i \leq \ell$ is the mapping cone of the map  $\phi_i \colon C_i \to P_{i-1}$ defined in \eqref{eq_attaching_map}.
\end{enumerate}

For convenience, we denote by $\mathscr{C}_{\{\mc{P}_i\}}$ the sequence $\{C_2, \dots, C_{\ell}\}$ 
corresponding to the given conic sequence $\{\mc{P}_i\}_{1\leq i \leq \ell}$ of a polytope $P$. If  $\mathscr{C}_{\{\mc{P}_i\}}$ consists of simplices, we call \eqref{eq_conic_decomp} a \emph{$\Delta$-conic sequence}, which agrees with the definition of a retraction sequence studied in \cite{BSS, BNSS}. A polytope admitting a $\Delta$-conic sequence  is referred to as an \emph{almost simple polytope} in \cite{SaSo2}.  We note that every simple polytope has at least one $\Delta$-conic sequence (see \cite[Proposition 2.3]{BSS}), which is not true for arbitrary convex polytope.

\begin{example}\label{ex_triangular_prism}
A triangular bipyramid $P$ is conic, but not $\Delta$-conic. Indeed, there is a conic sequence $\{\mc{P}_i\}_{1\leq i \leq 5}$ of $P$ such that  $\mathscr{C}_{\{\mc{P}_i\}}=\{\mathrm{pt}, I, I, I^2\}$, see Figure \ref{fig_dual_prism}. 
\end{example}
\begin{figure}
\tdplotsetmaincoords{80}{30}%{rotation around x axis}{rotation around z axis}
\begin{tikzpicture}[tdplot_main_coords, yscale=1]

\filldraw (0,0,-0.8) circle (1pt);
\node  at (0,0,-1.5) {$P_1$};

\begin{scope}[xshift=60]
\filldraw[red] (-1,0,0) circle (1.5pt);
\draw[thick, red] (0,0,-0.8)--(-1,0,0);
\filldraw (0,0,-0.8) circle (1pt);
\node  at (0,0,-1.5) {${P_2\atop C_2=\text{pt}}$};

\end{scope}

\begin{scope}[xshift=120]
\draw[fill=red, opacity=0.3] (0,0,-0.8)--(-1,0,0)--(1,-2,0.2)--cycle;
\draw (0,0,-0.8)--(-1,0,0);
\draw[thick, red] (1,-2,0.2)--(0,0,-0.8);
\draw[thick, red] (1,-2,0.2)--(-1,0,0);
\filldraw (0,0,-0.8) circle (1pt);
\filldraw (-1,0,0) circle (1pt);
\filldraw[red] (1,-2,0.2) circle (1.5pt);
\node  at (0,0,-1.5) {${P_3\atop C_3=I}$};

\end{scope}

\begin{scope}[xshift=180]
\draw[fill=blue!30] (0,0,-0.8)--(-1,0,0)--(1,-2,0.2)--cycle;
\draw (0,0,-0.8)--(-1,0,0)--(1,-2,0.2)--cycle;
\draw (0,0,-0.8)--(-1,0,0);
\draw[fill=red, opacity=0.3] (0,0,0.8)--(-1,0,0)--(1,-2,0.2)--cycle;
\draw[thick, red] (1,-2,0.2)--(0,0,0.8);
\draw[thick, red] (-1,0,0)--(0,0,0.8);
\filldraw (0,0,-0.8) circle (1pt);
\filldraw (-1,0,0) circle (1pt);
\filldraw (1,-2,0.2) circle (1pt);
\filldraw[red] (0,0,0.8) circle (1.5pt);
\node  at (0,0,-1.5) {${P_4\atop C_4=I}$};

\end{scope}

\begin{scope}[xshift=240]
\draw[fill=red, opacity=0.3] (0,0,0.8)--(-1,0,0)--(0,0,-0.8)--(0,1,0)--cycle;
\draw[fill=blue!30] (0,0,0.8)--(-1,0,0)--(0,0,-0.8)--(1,-2,0.2)--cycle;
\draw (-1,0,0)--(1,-2,0.2)--(0,1,0);
\draw (0,0,0.8)--(-1,0,0)--(0,0,-0.8);
\draw (0,0,0.8)--(1,-2,0.2)--(0,0,-0.8);
\draw (0,0,0.8)--(0,1,0)--(0,0,-0.8);

\draw[red, thick, dashed] (0,1,0)--(-1,0,0); 
\draw[thick, red] (0,1,0)--(0,0,0.8);
\draw[thick, red] (0,1,0)--(0,0,-0.8);
\draw[thick, red] (0,1,0)--(1,-2,0.2);

\filldraw (-1,0,0) circle (1pt);
\filldraw[red] (0,1,0) circle (1.5pt);
\filldraw (0,0,0.8) circle (1pt);
\filldraw (0,0,-0.8) circle (1pt);
\filldraw (1,-2,0.2) circle (1pt);
\node  at (0,0,-1.5) {${P_5\atop C_5=I^2}$};

\end{scope}

\end{tikzpicture}
\caption{A conic sequence on a triangular bipyramid.}
\label{fig_dual_prism}
\end{figure}
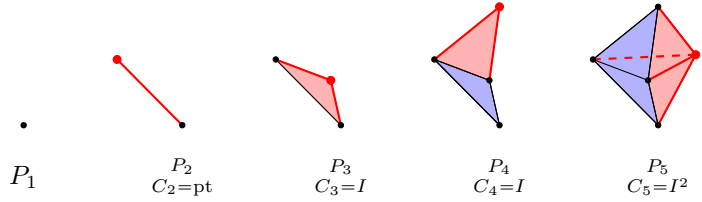

In the following, we see the relationship between the face numbers of $P$ and the face numbers of $C_i\in \mathscr{C}_{\{\mc{P}_i\}}$ for $i=2, \dots, \ell$. Let $f_i(P)$ be the number of $i$-dimensional faces of a polytope $P$, and let $\Phi_P(x)$ be the generating function of the face numbers of $P$, i.e., $\Phi_P(x)\colonequals \sum_{j=0}^{\dim P}f_j(P)x^j$.

\begin{proposition}\label{prop_f_vec}
Let $\mathscr{C}_{\{\mc{P}_i\}}=\{C_2, \dots, C_{\ell}\}$ be the sequence of polytopes corresponding to a conic sequence $\{\mc{P}_i\}_{1\leq i \leq \ell}$ of $P$ as in \eqref{eq_conic_decomp}.  Then  
\begin{equation*}%\label{eq_f_vec}
\Phi_P(x)=1+\sum_{i=2}^{\ell} \left(1+x\Phi_{C_i}(x)\right).
\end{equation*}
\end{proposition}
\begin{proof}
Applying \eqref{eq_conic_seq_next_term} iteratively, we obtain 
\begin{equation}\label{eq_f_vec_conic}
\Phi_P(x)= 1+\sum_{i=2}^{\ell} \sum_{F\in [v_i, E_i]} x^{\dim F}.
\end{equation}
Since there is a one-to-one correspondence between the $k$-dimensional faces of $C_i$ and the $(k+1)$-dimensional faces of the interval $[v_i, E_i]$ for each $0 \leq k\leq \dim C_i$, we have 
\begin{equation}\label{eq_f_vec_cC_i}
 \sum_{F\in [v_i, E_i]} x^{\dim F} = 1+ \sum_{k=0}^{\dim C_i} f_k(C_i)x^{k+1} = 1+x\Phi_{C_k}(x). 
\end{equation}
Hence the result follows by combining \eqref{eq_f_vec_conic} and \eqref{eq_f_vec_cC_i}. 
\end{proof}

\begin{example}\label{ex_BIP}
Let $P$ be the following Bruhat interval polytope (see \cite{TsWi}):
\[
\mathsf{Q}_{1324, 4231}\colonequals \text{conv}\{(\sigma(1),\sigma(2),\sigma(3),\sigma(4)) \in \RR^4 \mid  \sigma\in \mathfrak{S}_4,~ (1324) \leq \sigma  \leq (4231)\},
\]
where $\leq$ denotes the Bruhat order on the permutation group $\mathfrak{S}_4$ and $(i_1 i_2 i_3 i_4)$ is a permutation written in the one-line notation. See Figure \ref{fig_BIP} for the description of $\mathsf{Q}_{1324, 4231}$. A conic sequence of $\mathsf{Q}_{1324, 4231}$ is illustrated in Figure \ref{fig_conic_decomp_BIP}. In this case, the corresponding sequence  $\mathscr{C}_{\{\mc{P}_i\}}$ and  their generating functions are given by  
\begin{equation*}%\label{eq_cone_sp_for_BIP}
C_i=\begin{cases} \mathrm{pt} & 2\leq i \leq 5; \\
I & 6\leq i \leq 15;\\
I^2 &  i=16,
\end{cases}\quad \text{and} \quad 
\Phi_k(x)= \begin{cases} 1& 2 \leq k \leq 5;\\
2+x & 6 \leq k \leq 15;\\ 
4+4x+x^2 & k=16.
\end{cases}
\end{equation*}
We apply Proposition \ref{prop_f_vec} to get 
\begin{align*}%\label{eq_fvec_of_BIP}
\begin{split}
\sum_{k=0}^3 f_k(\mathsf{Q}_{1324, 4231}) x^k &=1+ 4(1+x) + 10(1+2x+x^2)  + (1+4x +4x^2 + x^3) \\
&= 16+28x+14x^2+x^3. 
\end{split}
\end{align*}
It is easy to see that $\mathsf{Q}_{1324, 4231}$ does not admit a $\Delta$-conic sequence, see \cite[Section 5]{LMPS} for instance. 
\end{example}

\begin{figure}
\begin{tikzpicture}

\draw[blue, thick, dashed] (1,4,2,3)--(1,3,2,4)--(2,3,1,4)--(2,4,1,3);
\draw[blue, thick, dashed] (4,2,1,3)--(3,2,1,4)--(3,1,2,4)--(4,1,2,3);

\node at (2,3,1,4) {\tiny$(2,3,1,4)$};
\node[above] at (2,4,1,3) {\tiny$(2,4,1,3)$};
\node[above] at (1,4,2,3) {\tiny$(1,4,2,3)$};
\node[left] at (1,3,2,4) {\tiny$(1,3,2,4)$};

\node[left] at (1,3,4,2) {\tiny$(1,3,4,2)$};
\node[left] at (1,4,3,2) {\tiny$(1,4,3,2)$};

\node[below] at (3,1,4,2) {\tiny$(3,1,4,2)$};

\node[right] at (4,2,1,3) {\tiny$(4,2,1,3)$};
\node[right] at (4,1,2,3) {\tiny$(4,1,2,3)$};
\node at (3,1,2,4) {\tiny$(3,1,2,4)$};
\node at (3,2,1,4) {\tiny$(3,2,1,4)$};

\node[left] at (1,2,3,4) {\tiny$(1,2,3,4)$};
\node[left] at (1,2,4,3) {\tiny$(1,2,4,3)$};
\node[left] at (2,1,4,3) {\tiny$(2,1,4,3)$};
\node[left] at (2,1,3,4) {\tiny$(2,1,3,4)$};

\node[right] at (4,3,2,1) {\tiny$(4,3,2,1)$};
\node[right] at (3,4,2,1) {\tiny$(3,4,2,1)$};
\node[right] at (3,4,1,2) {\tiny$(3,4,1,2)$};
\node[right] at (4,3,1,2) {\tiny$(4,3,1,2)$};

\draw[blue, thick, dashed] (3,1,4,2)--(3,1,2,4);
\draw[blue, thick, dashed] (1,3,4,2)--(1,3,2,4);
\draw[blue, thick, dashed] (1,3,2,4)--(3,1,2,4);

\draw[blue, thick, dashed] (3,2,1,3)--(2,3,1,4);

\draw[dotted] (1,2,3,4)--(1,2,4,3)--(2,1,4,3)--(2,1,3,4)--cycle;
\draw (1,2,4,3)--(2,1,4,3);

\draw[dotted] (1,2,3,4)--(1,3,2,4);
\draw (1,2,4,3)--(1,3,4,2);
\draw (2,1,4,3)--(3,1,4,2);
\draw[dotted] (2,1,3,4)--(3,1,2,4);

\draw (3,4,1,2)--(3,4,2,1)--(4,3,2,1)--(4,3,1,2)--cycle;
\draw (3,4,1,2)--(2,4,1,3);
\draw (4,3,1,2)--(4,2,1,3);

\filldraw (1,3,2,4) circle (2pt);

\draw[fill=blue!20, blue!20, opacity=0.6] (2,4,1,3)--(1,4,2,3)--(1,4,3,2)--(1,3,4,2)--(3,1,4,2)--(4,1,3,2)--(4,1,2,3)--(4,2,1,3)--cycle;

\draw (4,3,2,1)--(4,2,3,1);
\draw (3,4,2,1)--(2,4,3,1);
\draw[very thick, blue] (2,4,1,3)--(1,4,2,3);
\draw[very thick, blue]  (1,4,2,3)--(1,4,3,2);
\draw[very thick, blue]  (2,3,4,1)--(3,2,4,1);
\draw[very thick, blue]  (4,1,3,2)--(4,1,2,3);
\draw[very thick, blue] (2,4,1,3)--(4,2,1,3);
\draw[very thick, blue]  (2,4,3,1)--(2,3,4,1)--(1,3,4,2)--(1,4,3,2)--cycle;
\draw[very thick, blue]  (4,2,3,1)--(3,2,4,1)--(3,1,4,2)--(4,1,3,2)--cycle;
\draw[very thick, blue]  (4,2,1,3)--(4,1,2,3);
\draw[very thick, blue]  (2,4,1,3)--(2,4,3,1)--(4,2,3,1)--(4,2,1,3);
\draw[very thick, blue]  (1,3,4,2)--(3,1,4,2);

\node[right] at (2,3,4,1) {\tiny$(2,3,4,1)$};
\node[above] at (3,2,4,1) {\tiny$(3,2,4,1)$};
\node[right] at (4,2,3,1) {\tiny$(4,2,3,1)$};
\node[below] at (4,1,3,2) {\tiny$(4,1,3,2)$};
\node [above] at (2,4,3,1) {\tiny$(2,4,3,1)$};

\filldraw (4,2,3,1) circle (2pt);

\end{tikzpicture}
\caption{The Bruhat interval polytope $\mathsf{Q}_{1324, 4231}$.}
\label{fig_BIP}
\end{figure}
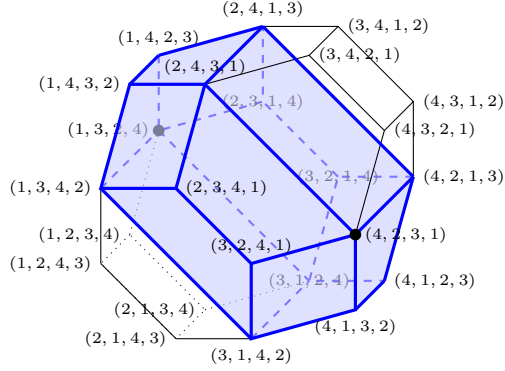

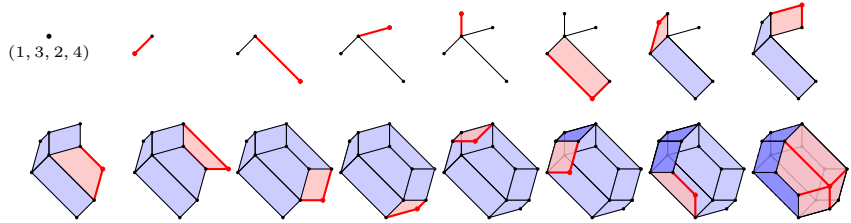
\begin{figure}
\begin{tikzpicture}[scale=0.3]

\filldraw (1,3,2,4) circle (2.5pt);
\node[below] at (1,3,2,4) {\tiny$(1,3,2,4)$};

\begin{scope}[xshift=130]
\draw[thick, red] (1,3,2,4)--(1,3,4,2);
\filldraw (1,3,2,4) circle (1.5pt);
\filldraw[red] (1,3,4,2) circle (2.5pt);
\end{scope}

\begin{scope}[xshift=260]
\draw (1,3,2,4)--(1,3,4,2);
\draw[thick, red] (1,3,2,4)--(3,1,2,4);
\filldraw (1,3,2,4) circle (1.5pt);
\filldraw (1,3,4,2) circle (1.5pt);
\filldraw[red] (3,1,2,4) circle (2.5pt);
\end{scope}

\begin{scope}[xshift=390]
\draw (1,3,2,4)--(1,3,4,2);
\draw (1,3,2,4)--(3,1,2,4);
\draw[thick, red] (1,3,2,4)--(2,3,1,4);

\filldraw (1,3,2,4) circle (1.5pt);
\filldraw (1,3,4,2) circle (1.5pt);
\filldraw (3,1,2,4) circle (1.5pt);
\filldraw[red] (2,3,1,4) circle (2.5pt);
\end{scope}

\begin{scope}[xshift=520]
\draw (1,3,2,4)--(1,3,4,2);
\draw (1,3,2,4)--(3,1,2,4);
\draw (1,3,2,4)--(2,3,1,4);
\draw[thick, red] (1,3,2,4)--(1,4,2,3);
\filldraw (1,3,2,4) circle (1.5pt);
\filldraw (1,3,4,2) circle (1.5pt);
\filldraw (3,1,2,4) circle (1.5pt);
\filldraw (2,3,1,4) circle (1.5pt);
\filldraw[red] (1,4,2,3) circle (2.5pt);
\end{scope}

\begin{scope}[xshift=650]
\draw[fill, red!20] (1,3,4,2)--(3,1,4,2)--(3,1,2,4)--(1,3,2,4)--cycle;
\draw[thick, red] (1,3,4,2)--(3,1,4,2)--(3,1,2,4);

\draw (1,3,2,4)--(1,3,4,2);
\draw (1,3,2,4)--(3,1,2,4);
\draw (1,3,2,4)--(2,3,1,4);
\draw (1,3,2,4)--(1,4,2,3);

\filldraw (1,3,2,4) circle (1.5pt);
\filldraw[red] (3,1,4,2) circle (2.5pt);
\filldraw (1,3,4,2) circle (1.5pt);
\filldraw (3,1,2,4) circle (1.5pt);
\filldraw (2,3,1,4) circle (1.5pt);
\filldraw (1,4,2,3) circle (1.5pt);
\end{scope}

\begin{scope}[xshift=780]
\draw[fill, blue!20] (1,3,4,2)--(3,1,4,2)--(3,1,2,4)--(1,3,2,4)--cycle;
\draw (1,3,4,2)--(3,1,4,2)--(3,1,2,4)--(1,3,2,4)--cycle;

\draw[fill, red!20] (1,4,3,2)--(1,4,2,3)--(1,3,2,4)--(1,3,4,2)--cycle;
\draw[thick, red] (1,4,2,3)--(1,4,3,2)--(1,3,4,2);

\filldraw (1,3,2,4) circle (1.5pt);
\draw (1,3,2,4)--(1,3,4,2);
\draw (1,3,2,4)--(3,1,2,4);
\draw (1,3,2,4)--(2,3,1,4);
\draw (1,3,2,4)--(1,4,2,3);
\filldraw (1,3,4,2) circle (1.5pt);
\filldraw (3,1,2,4) circle (1.5pt);
\filldraw (2,3,1,4) circle (1.5pt);
\filldraw (1,4,2,3) circle (1.5pt);
\filldraw (3,1,4,2) circle (1.5pt);
\filldraw[red] (1,4,3,2) circle (2.5pt);

\end{scope}

\begin{scope}[xshift=910]
\draw[fill, red!20] (1,3,2,4)--(1,4,2,3)--(2,4,1,3)--(2,3,1,4)--cycle;
\draw[thick, red]  (1,4,2,3)--(2,4,1,3)--(2,3,1,4);

\draw[fill, blue!20] (1,3,4,2)--(3,1,4,2)--(3,1,2,4)--(1,3,2,4)--cycle;
\draw (1,3,4,2)--(3,1,4,2)--(3,1,2,4)--(1,3,2,4)--cycle;

\draw[fill, blue!20] (1,4,3,2)--(1,4,2,3)--(1,3,2,4)--(1,3,4,2)--cycle;
\draw (1,4,3,2)--(1,4,2,3)--(1,3,2,4)--(1,3,4,2)--cycle;

\filldraw (1,3,2,4) circle (1.5pt);
\draw (1,3,2,4)--(1,3,4,2);
\draw (1,3,2,4)--(3,1,2,4);
\draw (1,3,2,4)--(2,3,1,4);
\draw (1,3,2,4)--(1,4,2,3);
\filldraw (1,3,4,2) circle (1.5pt);
\filldraw (3,1,2,4) circle (1.5pt);
\filldraw (2,3,1,4) circle (1.5pt);
\filldraw (1,4,2,3) circle (1.5pt);
\filldraw (3,1,4,2) circle (1.5pt);
\filldraw (1,4,3,2) circle (1.5pt);

\filldraw[red] (2,4,1,3) circle (2.5pt);

\end{scope}

\begin{scope}[yshift=-150]
\draw[fill, red!20] (1,3,2,4)--(2,3,1,4)--(3,2,1,4)--(3,1,2,4)--cycle;
\draw[thick, red] (2,3,1,4)--(3,2,1,4)--(3,1,2,4);

\draw[fill, blue!20] (1,3,2,4)--(1,4,2,3)--(2,4,1,3)--(2,3,1,4)--cycle;
\draw (1,4,2,3)--(2,4,1,3)--(2,3,1,4);
\filldraw (2,4,1,3) circle (1.5pt);

\draw[fill, blue!20] (1,3,4,2)--(3,1,4,2)--(3,1,2,4)--(1,3,2,4)--cycle;
\draw (1,3,4,2)--(3,1,4,2)--(3,1,2,4)--(1,3,2,4)--cycle;
\filldraw (3,1,4,2) circle (1.5pt);

\draw[fill, blue!20] (1,4,3,2)--(1,4,2,3)--(1,3,2,4)--(1,3,4,2)--cycle;
\draw (1,4,3,2)--(1,4,2,3)--(1,3,2,4)--(1,3,4,2)--cycle;
\filldraw (1,4,3,2) circle (1.5pt);

\filldraw (1,3,2,4) circle (1.5pt);
\draw (1,3,2,4)--(1,3,4,2);
\draw (1,3,2,4)--(3,1,2,4);
\draw (1,3,2,4)--(2,3,1,4);
\draw (1,3,2,4)--(1,4,2,3);
\filldraw (1,3,4,2) circle (1.5pt);
\filldraw (3,1,2,4) circle (1.5pt);
\filldraw (2,3,1,4) circle (1.5pt);
\filldraw (1,4,2,3) circle (1.5pt);

\filldraw[red] (3,2,1,4) circle (2.5pt);

\end{scope}

\begin{scope}[yshift=-150, xshift=130]
\draw[fill, red!20] (2,4,1,3)--(4,2,1,3)--(3,2,1,4)--(2,3,1,4)--cycle;
\draw[red, thick]  (2,4,1,3)--(4,2,1,3)--(3,2,1,4)--(2,3,1,4);

\draw[fill, blue!20] (1,3,2,4)--(2,3,1,4)--(3,2,1,4)--(3,1,2,4)--cycle;
\draw (2,3,1,4)--(3,2,1,4)--(3,1,2,4);

\draw[fill, blue!20] (1,3,2,4)--(1,4,2,3)--(2,4,1,3)--(2,3,1,4)--cycle;
\draw (1,4,2,3)--(2,4,1,3)--(2,3,1,4);

\draw[fill, blue!20] (1,3,4,2)--(3,1,4,2)--(3,1,2,4)--(1,3,2,4)--cycle;
\draw (1,3,4,2)--(3,1,4,2)--(3,1,2,4)--(1,3,2,4)--cycle;

\draw[fill, blue!20] (1,4,3,2)--(1,4,2,3)--(1,3,2,4)--(1,3,4,2)--cycle;
\draw (1,4,3,2)--(1,4,2,3)--(1,3,2,4)--(1,3,4,2)--cycle;

\filldraw (1,3,2,4) circle (1.5pt);
\draw (1,3,2,4)--(1,3,4,2);
\draw (1,3,2,4)--(3,1,2,4);
\draw (1,3,2,4)--(2,3,1,4);
\draw (1,3,2,4)--(1,4,2,3);
\filldraw (1,3,4,2) circle (1.5pt);
\filldraw (3,1,2,4) circle (1.5pt);
\filldraw (2,3,1,4) circle (1.5pt);
\filldraw (1,4,2,3) circle (1.5pt);
\filldraw (2,4,1,3) circle (1.5pt);
\filldraw (3,1,4,2) circle (1.5pt);
\filldraw (1,4,3,2) circle (1.5pt);
\filldraw[red] (4,2,1,3) circle (2.5pt);

\filldraw (3,2,1,4) circle (1.5pt);

\end{scope}

\begin{scope}[yshift=-150, xshift=260]
\draw[fill, red!20] (3,2,1,4)--(4,2,1,3)--(4,1,2,3)--(3,1,2,4)--cycle;
\draw[red, thick] (4,2,1,3)--(4,1,2,3)--(3,1,2,4);

\draw[fill, blue!20] (1,3,2,4)--(2,3,1,4)--(3,2,1,4)--(3,1,2,4)--cycle;
\draw (2,3,1,4)--(3,2,1,4)--(3,1,2,4);

\draw[fill, blue!20] (1,3,2,4)--(1,4,2,3)--(2,4,1,3)--(2,3,1,4)--cycle;
\draw (1,4,2,3)--(2,4,1,3)--(2,3,1,4);
\filldraw (2,4,1,3) circle (1.5pt);

\draw[fill, blue!20] (1,3,4,2)--(3,1,4,2)--(3,1,2,4)--(1,3,2,4)--cycle;
\draw (1,3,4,2)--(3,1,4,2)--(3,1,2,4)--(1,3,2,4)--cycle;
\filldraw (3,1,4,2) circle (1.5pt);

\draw[fill, blue!20] (1,4,3,2)--(1,4,2,3)--(1,3,2,4)--(1,3,4,2)--cycle;
\draw (1,4,3,2)--(1,4,2,3)--(1,3,2,4)--(1,3,4,2)--cycle;
\filldraw (1,4,3,2) circle (1.5pt);

\draw[fill, blue!20] (2,4,1,3)--(4,2,1,3)--(3,2,1,4)--(2,3,1,4)--cycle;
\draw  (2,4,1,3)--(4,2,1,3)--(3,2,1,4)--(2,3,1,4)--cycle;
\filldraw (4,2,1,3) circle (1.5pt);

\filldraw (1,3,2,4) circle (1.5pt);
\draw (1,3,2,4)--(1,3,4,2);
\draw (1,3,2,4)--(3,1,2,4);
\draw (1,3,2,4)--(2,3,1,4);
\draw (1,3,2,4)--(1,4,2,3);
\filldraw (1,3,4,2) circle (1.5pt);
\filldraw (3,1,2,4) circle (1.5pt);
\filldraw (2,3,1,4) circle (1.5pt);
\filldraw (1,4,2,3) circle (1.5pt);
\filldraw (3,2,1,4) circle (1.5pt);

\filldraw[red] (4,1,2,3) circle (2.5pt);

\end{scope}

\begin{scope}[yshift=-150, xshift=390]
\draw[fill, red!20] (3,1,2,4)--(4,1,2,3)--(4,1,3,2)--(3,1,4,2)--cycle;
\draw[thick, red] (4,1,2,3)--(4,1,3,2)--(3,1,4,2);

\draw[fill, blue!20] (1,3,2,4)--(2,3,1,4)--(3,2,1,4)--(3,1,2,4)--cycle;
\draw (2,3,1,4)--(3,2,1,4)--(3,1,2,4);

\draw[fill, blue!20] (1,3,2,4)--(1,4,2,3)--(2,4,1,3)--(2,3,1,4)--cycle;
\draw (1,4,2,3)--(2,4,1,3)--(2,3,1,4);
\filldraw (2,4,1,3) circle (1.5pt);

\draw[fill, blue!20] (1,3,4,2)--(3,1,4,2)--(3,1,2,4)--(1,3,2,4)--cycle;
\draw (1,3,4,2)--(3,1,4,2)--(3,1,2,4)--(1,3,2,4)--cycle;
\filldraw (3,1,4,2) circle (1.5pt);

\draw[fill, blue!20] (1,4,3,2)--(1,4,2,3)--(1,3,2,4)--(1,3,4,2)--cycle;
\draw (1,4,3,2)--(1,4,2,3)--(1,3,2,4)--(1,3,4,2)--cycle;
\filldraw (1,4,3,2) circle (1.5pt);

\draw[fill, blue!20] (1,4,3,2)--(1,4,2,3)--(1,3,2,4)--(1,3,4,2)--cycle;
\draw (1,4,3,2)--(1,4,2,3)--(1,3,2,4)--(1,3,4,2)--cycle;
\filldraw (1,4,3,2) circle (1.5pt);

\draw[fill, blue!20] (2,4,1,3)--(4,2,1,3)--(3,2,1,4)--(2,3,1,4)--cycle;
\draw  (2,4,1,3)--(4,2,1,3)--(3,2,1,4)--(2,3,1,4)--cycle;
\filldraw (4,2,1,3) circle (1.5pt);

\draw[fill, blue!20] (3,2,1,4)--(4,2,1,3)--(4,1,2,3)--(3,1,2,4)--cycle;
\draw (3,2,1,4)--(4,2,1,3)--(4,1,2,3)--(3,1,2,4)--cycle;
\filldraw (4,1,2,3) circle (1.5pt);

\filldraw (1,3,2,4) circle (1.5pt);
\draw (1,3,2,4)--(1,3,4,2);
\draw (1,3,2,4)--(3,1,2,4);
\draw (1,3,2,4)--(2,3,1,4);
\draw (1,3,2,4)--(1,4,2,3);
\filldraw (1,3,4,2) circle (1.5pt);
\filldraw (3,1,2,4) circle (1.5pt);
\filldraw (2,3,1,4) circle (1.5pt);
\filldraw (1,4,2,3) circle (1.5pt);

\filldraw (3,2,1,4) circle (1.5pt);
\filldraw[red] (4,1,3,2) circle (2.5pt);

\end{scope}

\begin{scope}[yshift=-150, xshift=520]
\draw[fill, blue!20] (1,3,2,4)--(2,3,1,4)--(3,2,1,4)--(3,1,2,4)--cycle;
\draw (2,3,1,4)--(3,2,1,4)--(3,1,2,4);

\draw[fill, blue!20] (1,3,2,4)--(1,4,2,3)--(2,4,1,3)--(2,3,1,4)--cycle;
\draw (1,4,2,3)--(2,4,1,3)--(2,3,1,4);
\draw[fill, blue!20] (1,3,4,2)--(3,1,4,2)--(3,1,2,4)--(1,3,2,4)--cycle;
\draw (1,3,4,2)--(3,1,4,2)--(3,1,2,4)--(1,3,2,4)--cycle;

\draw[fill, blue!20] (1,4,3,2)--(1,4,2,3)--(1,3,2,4)--(1,3,4,2)--cycle;
\draw (1,4,3,2)--(1,4,2,3)--(1,3,2,4)--(1,3,4,2)--cycle;

\draw[fill, blue!20] (1,4,3,2)--(1,4,2,3)--(1,3,2,4)--(1,3,4,2)--cycle;
\draw (1,4,3,2)--(1,4,2,3)--(1,3,2,4)--(1,3,4,2)--cycle;

\draw[fill, blue!20] (2,4,1,3)--(4,2,1,3)--(3,2,1,4)--(2,3,1,4)--cycle;
\draw  (2,4,1,3)--(4,2,1,3)--(3,2,1,4)--(2,3,1,4)--cycle;

\draw[fill, blue!20] (3,2,1,4)--(4,2,1,3)--(4,1,2,3)--(3,1,2,4)--cycle;
\draw (3,2,1,4)--(4,2,1,3)--(4,1,2,3)--(3,1,2,4)--cycle;

\draw[fill, blue!20] (3,1,2,4)--(4,1,2,3)--(4,1,3,2)--(3,1,4,2)--cycle;
\draw (3,1,2,4)--(4,1,2,3)--(4,1,3,2)--(3,1,4,2)--cycle;

\draw (1,3,2,4)--(1,3,4,2);
\draw (1,3,2,4)--(3,1,2,4);
\draw (1,3,2,4)--(2,3,1,4);
\draw (1,3,2,4)--(1,4,2,3);
\filldraw (1,3,2,4) circle (1.5pt);

\draw[fill=pink, opacity=0.8] (1,4,3,2)--(2,4,3,1)--(2,4,1,3)--(1,4,2,3)--cycle;
\draw[red, thick] (1,4,3,2)--(2,4,3,1)--(2,4,1,3);

\filldraw (1,3,4,2) circle (1.5pt);
\filldraw (3,1,2,4) circle (1.5pt);
\filldraw (2,3,1,4) circle (1.5pt);
\filldraw (1,4,2,3) circle (1.5pt);

\filldraw (3,2,1,4) circle (1.5pt);
\filldraw (4,2,1,3) circle (1.5pt);
\filldraw (4,1,2,3) circle (1.5pt);
\filldraw (2,4,1,3) circle (1.5pt);
\filldraw (3,1,4,2) circle (1.5pt);
\filldraw (1,4,3,2) circle (1.5pt);
\filldraw (1,4,3,2) circle (1.5pt);
\filldraw (4,1,3,2) circle (1.5pt);

\filldraw[red] (2,4,3,1) circle (2.5pt);

\end{scope}

\begin{scope}[yshift=-150, xshift=650]
\draw[fill, blue!20] (1,3,2,4)--(2,3,1,4)--(3,2,1,4)--(3,1,2,4)--cycle;
\draw (2,3,1,4)--(3,2,1,4)--(3,1,2,4);

\draw[fill, blue!20] (1,3,2,4)--(1,4,2,3)--(2,4,1,3)--(2,3,1,4)--cycle;
\draw (1,4,2,3)--(2,4,1,3)--(2,3,1,4);

\draw[fill, blue!20] (1,3,4,2)--(3,1,4,2)--(3,1,2,4)--(1,3,2,4)--cycle;
\draw (1,3,4,2)--(3,1,4,2)--(3,1,2,4)--(1,3,2,4)--cycle;

\draw[fill, blue!20] (1,4,3,2)--(1,4,2,3)--(1,3,2,4)--(1,3,4,2)--cycle;
\draw (1,4,3,2)--(1,4,2,3)--(1,3,2,4)--(1,3,4,2)--cycle;

\draw[fill, blue!20] (1,4,3,2)--(1,4,2,3)--(1,3,2,4)--(1,3,4,2)--cycle;
\draw (1,4,3,2)--(1,4,2,3)--(1,3,2,4)--(1,3,4,2)--cycle;

\draw[fill, blue!20] (2,4,1,3)--(4,2,1,3)--(3,2,1,4)--(2,3,1,4)--cycle;
\draw  (2,4,1,3)--(4,2,1,3)--(3,2,1,4)--(2,3,1,4)--cycle;

\draw[fill, blue!20] (3,2,1,4)--(4,2,1,3)--(4,1,2,3)--(3,1,2,4)--cycle;
\draw (3,2,1,4)--(4,2,1,3)--(4,1,2,3)--(3,1,2,4)--cycle;

\draw[fill, blue!20] (3,1,2,4)--(4,1,2,3)--(4,1,3,2)--(3,1,4,2)--cycle;
\draw (3,1,2,4)--(4,1,2,3)--(4,1,3,2)--(3,1,4,2)--cycle;

\draw (1,3,2,4)--(1,3,4,2);
\draw (1,3,2,4)--(3,1,2,4);
\draw (1,3,2,4)--(2,3,1,4);
\draw (1,3,2,4)--(1,4,2,3);
\filldraw (1,3,2,4) circle (1.5pt);

\draw[fill=blue!50, opacity=0.8] (1,4,3,2)--(2,4,3,1)--(2,4,1,3)--(1,4,2,3)--cycle;
\draw (1,4,3,2)--(2,4,3,1)--(2,4,1,3)--(1,4,2,3)--cycle;

\draw[fill=pink, opacity=0.8] (1,4,3,2)--(2,4,3,1)--(2,3,4,1)--(1,3,4,2)--cycle;
\draw[red, thick] (2,4,3,1)--(2,3,4,1)--(1,3,4,2);

\filldraw (1,3,4,2) circle (1.5pt);
\filldraw (3,1,2,4) circle (1.5pt);
\filldraw (2,3,1,4) circle (1.5pt);
\filldraw (1,4,2,3) circle (1.5pt);

\filldraw (3,2,1,4) circle (1.5pt);
\filldraw (4,2,1,3) circle (1.5pt);
\filldraw (4,1,2,3) circle (1.5pt);
\filldraw[red] (2,3,4,1) circle (2.5pt);
\filldraw (2,4,3,1) circle (1.5pt);
\filldraw (4,1,3,2) circle (1.5pt);
\filldraw (4,1,2,3) circle (1.5pt);
\filldraw (4,2,1,3) circle (1.5pt);
\filldraw (1,4,3,2) circle (1.5pt);
\filldraw (1,4,3,2) circle (1.5pt);
\filldraw (2,4,1,3) circle (1.5pt);
\filldraw (3,1,4,2) circle (1.5pt);
\end{scope}

\begin{scope}[yshift=-150, xshift=780]
\draw[fill, blue!20] (1,3,2,4)--(2,3,1,4)--(3,2,1,4)--(3,1,2,4)--cycle;
\draw (2,3,1,4)--(3,2,1,4)--(3,1,2,4);

\draw[fill, blue!20] (1,3,2,4)--(1,4,2,3)--(2,4,1,3)--(2,3,1,4)--cycle;
\draw (1,4,2,3)--(2,4,1,3)--(2,3,1,4);

\draw[fill, blue!20] (1,3,4,2)--(3,1,4,2)--(3,1,2,4)--(1,3,2,4)--cycle;
\draw (1,3,4,2)--(3,1,4,2)--(3,1,2,4)--(1,3,2,4)--cycle;

\draw[fill, blue!20] (1,4,3,2)--(1,4,2,3)--(1,3,2,4)--(1,3,4,2)--cycle;
\draw (1,4,3,2)--(1,4,2,3)--(1,3,2,4)--(1,3,4,2)--cycle;

\draw[fill, blue!20] (1,4,3,2)--(1,4,2,3)--(1,3,2,4)--(1,3,4,2)--cycle;
\draw (1,4,3,2)--(1,4,2,3)--(1,3,2,4)--(1,3,4,2)--cycle;

\draw[fill, blue!20] (2,4,1,3)--(4,2,1,3)--(3,2,1,4)--(2,3,1,4)--cycle;
\draw  (2,4,1,3)--(4,2,1,3)--(3,2,1,4)--(2,3,1,4)--cycle;

\draw[fill, blue!20] (3,2,1,4)--(4,2,1,3)--(4,1,2,3)--(3,1,2,4)--cycle;
\draw (3,2,1,4)--(4,2,1,3)--(4,1,2,3)--(3,1,2,4)--cycle;

\draw[fill, blue!20] (3,1,2,4)--(4,1,2,3)--(4,1,3,2)--(3,1,4,2)--cycle;
\draw (3,1,2,4)--(4,1,2,3)--(4,1,3,2)--(3,1,4,2)--cycle;

\filldraw (1,3,2,4) circle (1.5pt);
\draw (1,3,2,4)--(1,3,4,2);
\draw (1,3,2,4)--(3,1,2,4);
\draw (1,3,2,4)--(2,3,1,4);
\draw (1,3,2,4)--(1,4,2,3);

\draw[fill=blue!50, opacity=0.8] (1,4,3,2)--(2,4,3,1)--(2,4,1,3)--(1,4,2,3)--cycle;
\draw (1,4,3,2)--(2,4,3,1)--(2,4,1,3)--(1,4,2,3)--cycle;

\draw[fill=blue!50, opacity=0.8] (1,4,3,2)--(2,4,3,1)--(2,3,4,1)--(1,3,4,2)--cycle;
\draw (1,4,3,2)--(2,4,3,1)--(2,3,4,1)--(1,3,4,2)--cycle;

\draw[fill=pink, opacity=0.8] (3,1,4,2)--(3,2,4,1)--(2,3,4,1)--(1,3,4,2)--cycle;
\draw[thick, red] (3,1,4,2)--(3,2,4,1)--(2,3,4,1);

\filldraw (1,3,4,2) circle (1.5pt);
\filldraw (3,1,2,4) circle (1.5pt);
\filldraw (2,3,1,4) circle (1.5pt);
\filldraw (1,4,2,3) circle (1.5pt);

\filldraw[red] (3,2,4,1) circle (2.5pt);
\filldraw (2,3,4,1) circle (1.5pt);
\filldraw (2,4,3,1) circle (1.5pt);
\filldraw (4,1,3,2) circle (1.5pt);
\filldraw (4,1,2,3) circle (1.5pt);
\filldraw (4,2,1,3) circle (1.5pt);
\filldraw (1,4,3,2) circle (1.5pt);
\filldraw (1,4,3,2) circle (1.5pt);
\filldraw (3,1,4,2) circle (1.5pt);
\filldraw (2,4,1,3) circle (1.5pt);
\filldraw (3,2,1,4) circle (1.5pt);

\end{scope}

\begin{scope}[yshift=-150, xshift=910]
\draw[fill, blue!20, dotted] (1,3,2,4)--(2,3,1,4)--(3,2,1,4)--(3,1,2,4)--cycle;
\draw[dotted] (2,3,1,4)--(3,2,1,4)--(3,1,2,4);

\draw[fill, blue!20, dotted] (1,3,2,4)--(1,4,2,3)--(2,4,1,3)--(2,3,1,4)--cycle;
\draw[dotted] (1,4,2,3)--(2,4,1,3)--(2,3,1,4);

\draw[fill, blue!20, dotted] (1,3,4,2)--(3,1,4,2)--(3,1,2,4)--(1,3,2,4)--cycle;
\draw[dotted] (1,3,4,2)--(3,1,4,2)--(3,1,2,4)--(1,3,2,4)--cycle;

\draw[fill, blue!20,dotted] (1,4,3,2)--(1,4,2,3)--(1,3,2,4)--(1,3,4,2)--cycle;
\draw[dotted] (1,4,3,2)--(1,4,2,3)--(1,3,2,4)--(1,3,4,2)--cycle;

\draw[fill, blue!20, dotted] (1,4,3,2)--(1,4,2,3)--(1,3,2,4)--(1,3,4,2)--cycle;
\draw[dotted] (1,4,3,2)--(1,4,2,3)--(1,3,2,4)--(1,3,4,2)--cycle;

\draw[fill, blue!20, dotted] (2,4,1,3)--(4,2,1,3)--(3,2,1,4)--(2,3,1,4)--cycle;
\draw (2,4,1,3)--(4,2,1,3)--(3,2,1,4)--(2,3,1,4)--cycle;

\draw[fill, blue!20, dotted] (3,2,1,4)--(4,2,1,3)--(4,1,2,3)--(3,1,2,4)--cycle;
\draw (3,2,1,4)--(4,2,1,3)--(4,1,2,3)--(3,1,2,4)--cycle;

\draw[fill, blue!20, dotted] (3,1,2,4)--(4,1,2,3)--(4,1,3,2)--(3,1,4,2)--cycle;
\draw (3,1,2,4)--(4,1,2,3)--(4,1,3,2)--(3,1,4,2)--cycle;

\draw (1,3,2,4)--(1,3,4,2);
\draw (1,3,2,4)--(3,1,2,4);
\draw (1,3,2,4)--(2,3,1,4);
\draw (1,3,2,4)--(1,4,2,3);

\filldraw (3,2,1,4) circle (1.5pt);
\filldraw (2,3,1,4) circle (1.5pt);
\filldraw (3,1,2,4) circle (1.5pt);
\filldraw (1,3,2,4) circle (1.5pt);

\draw[fill=blue!50, opacity=0.8, dotted] (1,4,3,2)--(2,4,3,1)--(2,4,1,3)--(1,4,2,3)--cycle;
\draw (1,4,3,2)--(2,4,3,1)--(2,4,1,3)--(1,4,2,3)--cycle;

\draw[fill=blue!50, opacity=0.8, dotted] (1,4,3,2)--(2,4,3,1)--(2,3,4,1)--(1,3,4,2)--cycle;
\draw (1,4,3,2)--(2,4,3,1)--(2,3,4,1)--(1,3,4,2)--cycle;

\draw[fill=blue!50, opacity=0.8, dotted] (3,1,4,2)--(3,2,4,1)--(2,3,4,1)--(1,3,4,2)--cycle;
\draw (3,1,4,2)--(3,2,4,1)--(2,3,4,1)--(1,3,4,2)--cycle;

\draw[fill=pink, opacity=0.8] (3,2,4,1)--(3,1,4,2)--(4,1,3,2)--(4,1,2,3)--(4,2,1,3)--(2,4,1,3)--(2,4,3,1)--(2,3,4,1)--cycle;

\draw[red, thick] (4,2,3,1)--(3,2,4,1);
\draw[red, thick] (4,2,3,1)--(2,4,3,1);
\draw[red, thick] (4,2,3,1)--(4,2,1,3);
\draw[red, thick] (4,2,3,1)--(4,1,3,2);

\draw (2,4,1,3)--(4,2,1,3)--(4,1,2,3)--(4,1,3,2)--(3,1,4,2);

\filldraw (1,3,4,2) circle (1.5pt);
\filldraw (1,4,2,3) circle (1.5pt);
\filldraw (1,4,3,2) circle (1.5pt);

\filldraw (2,3,4,1) circle (1.5pt);
\filldraw (2,4,3,1) circle (1.5pt);
\filldraw (2,4,1,3) circle (1.5pt);

\filldraw (3,1,4,2) circle (1.5pt);
\filldraw (3,2,4,1) circle (1.5pt);

\filldraw (4,1,2,3) circle (1.5pt);
\filldraw (4,2,1,3) circle (1.5pt);
\filldraw (4,1,3,2) circle (1.5pt);
\filldraw[red] (4,2,3,1) circle (2.5pt);
\end{scope}

\end{tikzpicture}
\caption{A conic decomposition of $\mathsf{Q}_{1324, 4231}$.} 
\label{fig_conic_decomp_BIP}
\end{figure}

We notice that the generating function $\Phi_{\Delta^{d-1}}(x)$ of $(d-1)$-simplex $\Delta^{d-1}$ is 
$\sum_{k=1}^d {d \choose k}x^{k-1}$, which means 
\[
1+x\Phi_{\Delta^{d-1}}(x)=(x+1)^d.
\]
Hence Proposition \ref{prop_f_vec} implies that if $P$ is $\Delta$-conic, then 
\begin{align}\label{eq_fvec_hvec_rel}
    \sum_{k=0}^{\dim P} f_k(P)x^k = \sum_{k=0}^{\dim P}h_k(x+1)^k,
\end{align}
where $h_0=1$ and  $h_k$ denotes the number of $(k-1)$-simplices in  $\mathscr{C}_{\{\mc{P}_i\}}$ for a $\Delta$-conic sequence of $P$. Moreover, the condition $\dim C_i\leq \dim P_{i-1}$ for a conic sequence implies that $h_k\geq 1$ for all $k=1, \dots, \dim P$. This observation establishes the following corollary which provides a necessary condition for a polytope to be $\Delta$-conic. 
\begin{corollary}\label{cor_nec_cond}
If a convex polytope $P$ admits a $\Delta$-conic sequence, then $h_k$'s in~\eqref{eq_fvec_hvec_rel} are positive for all $k=1, \dots, \dim P$.
\end{corollary}

Using a similar idea of Corollary \ref{cor_nec_cond}, one can also discuss a necessary condition for a polytope $P$ to have a conic sequence such that the corresponding sequence $\mathscr{C}_{\{\mc{P}_i\}}$  consists of cubes. Indeed, since the generating function $\Phi_{I^{d-1}}(x)$ for the $(d-1)$-cube $I^{d-1}$ is 
\[
\Phi_{I^{d-1}}(x)=\sum_{k=1}^d {d-1 \choose k-1}2^{d-k}x^{k-1},
\]
we have 
\[
1+x\Phi_{I^{d-1}}(x)=1+x(x+2)^{d-1}.
\]
Therefore, if $P$ admits a conic sequence such that $\mathscr{C}_{\{\mc{P}_i\}}$ consists of cubes, Proposition \ref{prop_f_vec} gives us 
\begin{equation*}%\label{eq_fvec_hhvec_cube}
\sum_{k=0}^{\dim P} f_k(P)x^k= \sum_{k=0}^{\dim P} h_k^\square \left(1+x(x+2)^{k-1}\right),
\end{equation*}
for some positive integers $h_k^\square$'s with $h_1^\square=1$. 
 
We refer to  Example~\ref{ex_BIP}, where $h^\square(\mathsf{Q}_{1324, 4231})= (h^\square_0, h^\square_1, h^\square_2, h^\square_3)=(1,4,10,1)$. 

We remark that the same argument can apply to other conic sequences with $\mathscr{C}_{\{\mc{P}_i\}}$ consisting of certain polytopes whose generating functions are written as polynomials depending only on their dimensions like simplices or cubes.

\section{Conic decomposition of a toric variety}\label{sec_conic_decomp_X_P}
We begin this section with a topological model of a projective toric variety $X_P$ associated with a lattice polytope $P$, following \cite[Chapter 1]{Jor} which is attributed to MacPherson. This gives us a way to understand the underlying topological space of $X_P$. We also investigate several vanishing theorems for the singular cohomology of the toric variety $X_P$, associated with a conic polytope~$P$.  

\subsection{Topological model of a toric variety} Let $M$ be the character lattice of the $n$-dimensional compact torus $T^n$. We consider an $n$-dimensional lattice polytope~$P$ in $M_\RR\colonequals M\otimes_\ZZ \RR$. Let $\mathcal{F}(P)$ be the set of facets of $P$ and  $\lambda(F)$ be the primitive outward normal vector of $F\in \mathcal{F}(P)$. Regarding $\lambda(F)$ as an element in the lattice $N$ of one-parameter subgroups of $T^n$, we denote by $T_{F}$ the circle subgroup of $T^n$ generated by $\lambda(F)$. In general, given a face $E$ of $P$,
we denote by $N_E$ the submodule generated by 
\begin{equation}\label{eq_facets_intersecting_face}
\{\lambda(F)\in N \mid F\in \mathcal{F}(P),~F\cap E\neq \emptyset\}.
\end{equation}
Then $N_E$ determines the subtorus $T_E\colonequals N_E\otimes_\ZZ \RR/ \left((N_E\otimes_\ZZ \RR)\cap N\right)$ of $T^n$.

The topological model of the toric variety $X_P$ associated with a lattice polytope $P$ is given 
\begin{equation}\label{eq_colim_X_P}
X_P\colonequals (P \times T^n )/_\sim,
\end{equation}
where we identify $(p,t)\sim(q,s)$ whenever $p=q$ and $t^{-1}s \in T_E$ if $p$ is in the relative interior of $E$. Notice that the action of $T^n$ on $X_P$ given by the multiplication on the second factor induces the orbit map $\pi \colon X_P \to P$, that is the projection onto the first factor. 

If $P$ is a simple lattice polytope, then each face $E$ of codimension-$k$ in $P$ is the transversal intersection of $k$ facets of $P$ by definition. Hence the elements in \eqref{eq_facets_intersecting_face} are linearly independent. In this case, the resulting toric variety $X_P$ has at worst finite quotient singularities and $X_P$ is often called a toric orbifold, see \cite[Theorem 11.4.8]{CLS} for instance. We also refer to \cite[Section 7]{DJ} and  \cite{PS} for more rigorous topological interpretation of a toric orbifold. 

\subsection{Toric variety with a conic decomposition}
In what follows, we mainly consider a conic lattice polytope $P$ in $M_\RR$.  We may regard each $C_i\in \mathscr{C}_{\{\mc{P}_i\}}$ as a rational polytope lying in an affine subspace of $M_\RR$. This yields the following diagram, which we call a \emph{conic decomposition} of~$X_P$; 
\begin{equation}\label{eq_conic_on_X}
\begin{tikzcd}[column sep=small]
Y_2 \dar{\Phi_2} &  &  Y_i \dar{\Phi_i}   & Y_{i+1} \dar{\Phi_{i+1}}  &   &  Y_{\ell} \dar{\Phi_{\ell}}  &  \\
X_1 \rar[hook] & \cdots   \rar[hook]  &  X_{i-1}  \rar[hook]   & X_{i} \rar[hook] & \cdots \rar[hook]  & X_{\ell-1} \rar[hook]  & X_{\ell}=X, 
\end{tikzcd}\end{equation}
where $X_i=\pi^{-1}(P_i)$ and $Y_i=\pi^{-1}(C_i)$ with respect to the orbit map $\pi \colon X_P \to P$. Moreover,  the attaching map $\phi_i$ defined in \eqref{eq_attaching_map} induces a map $\Phi_i \colon Y_i \to X_{i-1}$ by the following commutative diagram: 
\[
\begin{tikzcd}
C_i \times T^n  \arrow{d}{\phi_i \times id} \rar[two heads]& (C_i \times T^n)/_{\sim}=Y_i   \arrow{d}{\Phi_i}  \\ P_{i-1} \times T^n \rar[two heads]  & (P_{i-1} \times T^n)/_\sim = X_{i-1}. 
\end{tikzcd}
\]
This leads us to understand each $X_i$ as the cofiber of  $\Phi_i \colon Y_i \to  X_{i-1}$.

The next two lemmas describe the topology of $Y_i$ when (i) $C_i$ is a simplex or (ii) $C_i$ is a simple polytope more generally.
\begin{lemma}\label{lem_Y_i_simplex}
If $C_i$ is a simplex $\Delta^{d_i-1}$ for some $d_i\geq 1$, then $Y_i$ is homeomorphic to $S^{2d_i-1}/K_i$ for some finite group $K_i$. 
\end{lemma}
\begin{proof}
The proof is essentially the same as the proof of \cite[Propsotion 4.4]{BNSS} and \cite[Corollary 3.7]{SaSo2}. However, for the reader’s convenience, we shall provide a sketch of the proof. Consider the conic decomposition in \eqref{eq_conic_on_X}, which  gives us 
\[
X_i-X_{i-1} = \pi^{-1}(P_i)-\pi^{-1}(P_{i-1}) \cong \pi^{-1}(U_{i}),
\]
where $U_{i}=P_{i}-P_{i-1}$ is homeomorphic to $\RR^{d_i}_\geq$ as a manifold with corners because $C_i=\Delta^{d_i-1}$. Then \eqref{eq_colim_X_P} implies that 
\begin{equation*}%\label{eq_pair_is_equiv_to_q-disk}
\pi^{-1}(U_{i}) \cong U_{i} \times T^{d_i}/_\sim \cong D^{2d_i}/K_i
\end{equation*}
for some finite subgroup $K_i$ of $T^{d_i+1}$ acting linearly on $D^{2d_i}$. Since $X_i$ is a mapping cone of the map $\Phi_i \colon Y_i \to  X_{i-1}$, we conclude that 
$Y_i\cong S^{2d_i-1}/K_i$ as desired.  
\end{proof}
We notice that the space $S^{2d_i-1}/K_i$ in Lemma \ref{lem_Y_i_simplex} above is called  an \emph{orbifold lens space} in \cite{BSS}. In Lemma \ref{lem_Y_i_for_simple_poly} below, we consider the case where $C_i$ is a simple polytope. We refer the readers to \cite{Rua} for the rigorous definition of an orbifold fiber bundle. 

\begin{lemma} \label{lem_Y_i_for_simple_poly}
If $C_i$ is a simple polytope, then $Y_i$ is an orbifold $S^1$-fiber bundle over a toric orbifold.
\end{lemma}
\begin{proof}
Let $E_i$ be the face of $P$ defined in \eqref{eq_conic_seq_next_term}. Then it is the minimal face of $P$ containing $C_i$. Since the equivalence relation $\sim$ given in \eqref{eq_colim_X_P} collapses $T_{E_i}$ globally on $C_i$, we may write 
\[
Y_i =(C_i \times T^n/T_{E_i})/_\sim. 
\]
Take an affine subspace $\mathcal{A}_{C_i}\subset M_\RR$ intersecting $E_i$ transversely such that $C_i$ is a rational polytope in $\mathcal{A}_{C_i}$ with respect to the lattice $\mathcal{A}_{C_i}\cap M$.  Let $\mu_{C_i}\in N\setminus N_{E_i}$ be a primitive element which determines $\mathcal{A}_{C_i}$. Then the submodule $N_{C_i}\subset N$, generated by all elements of $N_{E_i}$ together with $\mu_{C_i}$, gives rise to the torus 
\[
T_{C_i}\colonequals N_{C_i}\otimes_\ZZ \RR/\left( (N_{C_i} \otimes_\ZZ \RR)\cap N\right)
\]
with $\dim T_{C_i}=\dim T_{E_i}+1$. Accordingly, we have $\dim T^n/T_{E_i} = \dim T^n/T_{C_i}+1$. 

Let $S^1_{\mu_{C_i}}$ be the image of the composition 
\begin{equation}\label{eq_S1-composition}
S^1 \xrightarrow{\mu_{C_i}} T^n \twoheadrightarrow T^n/T_{E_i}
\end{equation}
regarding $\mu_{C_i}$ as a one-parameter subgroup of $T^n$. 
We consider the $S^1$-action on $Y_i$ given by \eqref{eq_S1-composition}. Then we claim that the corresponding orbit map 
\begin{equation*}%\label{eq_projection_S1_bdl}
Y_i \twoheadrightarrow Y_i/S^1_{\mu_{C_i}}.
\end{equation*}
is the desired orbifold $S^1$-fiber bundle over a toric orbifold. 
Since $\mu_{C_i}\in N\setminus N_{E_i}$, one can see that 
\begin{equation}\label{eq_toric_orb_Y_i_modulo_S1}
Y_i/S^1_{\mu_{C_i}}\cong (C_i\times T^n/T_{C_i})/_\sim,
\end{equation}
where the identification $\sim$ is the one induced from \eqref{eq_colim_X_P}. 
Since $C_i$ is a simple polytope by the assumption and $P_i$ is the mapping cone of the map $\phi_i \colon C_i \to P_{i-1}$ such that 
\[
P_i-P_{i-1}=\bigcup_{F\in [v_i, E_i]}\relint (F),\]
any vertex of $C_i$, say $w$, is the transversal intersection of $C_i$ and an edge of $E_i$, namely
\[w=C_i\cap  \bigcap_{j=1}^{\dim C_i} (E_i \cap F_{i_j})\]
for some facets $F_{i_j}$'s of $P$. This implies that the image \[
\{\rho_i(\lambda(F_{i_j}))\mid j=1, \dots, \dim C_i\}
\]
with respect to the projection $\rho_i \colon N \twoheadrightarrow N/N_{C_i}$ is linearly independent, and this is true for each vertex of $C_i$.  Therefore, $Y_i/S^1_{\mu_{C_i}}$ in~\eqref{eq_toric_orb_Y_i_modulo_S1} is a toric orbifold, which establishes the claim. 
\end{proof}

\section{Application to singular cohomology}
\label{sec_cohomology}
In this section, we prove several vanishing theorems for the singular cohomology of a toric variety $X_P$ with a conic decomposition. In particular, we prove that if $P$ admits a $\Delta$-conic sequence, then the (cohomology) Betti numbers of $X_P$ are completely determined by the face numbers of $P$. Throughout this section, we mean by $H^\ast(-)$ the singular cohomology with rational coefficients. 

\begin{theorem}\label{thm_main}
Let $X_P$ be a projective toric variety and $P$ the associated polytope. If there exists a conic sequence on  $P$ such that $\widetilde{H}^{2k}(Y_i)=0$ for all $i=2, \dots, \ell$, then $H^{2k+1}(X)=0$. 
\end{theorem}
\begin{proof}
Consider the conic decomposition in \eqref{eq_conic_on_X} induced from the given conic sequence  on $P$. We prove the claim by the induction on the indices of the following filtration
\begin{equation}\label{eq_filtration_on_X}
X_1 \hookrightarrow X_2 \hookrightarrow \cdots \hookrightarrow X_{\ell-1} \hookrightarrow X_\ell =X
\end{equation}
induced from the conic decomposition of $X_P$ as in \eqref{eq_conic_on_X}. 

First, since $P_1$ is a vertex of $P$, $X_1$ is zero-dimensional. Since $P_2$ is an edge of $P$, $X_2=\pi^{-1}(P_2)$ is homeomorphic to $S^2$, which means $H^{odd}(X_2)$ vanishes. We now claim that $H^{2k+1}(X_{i})=0$, provided $H^{2k+1}(X_{i-1})=0$ for $i>2$. 
Since each composition 
\begin{equation}\label{eq_cofibration}
Y_i \xrightarrow{\Phi_i} X_{i-1} \to X_i
\end{equation}
is a cofibration, we have the following long exact sequence of the cohomology groups
\begin{equation}\label{eq_les_from_cofib}
\cdots \to \widetilde H^{j-1}(Y_i) \to \widetilde H^j(X_i) \to \widetilde H^j(X_{i-1}) \to \widetilde H^j(Y_i) \to \cdots. 
\end{equation}
Thus, when $j=2k+1$ for some $k>0$, the assumption of $H^{2k}(Y_i)=0$ together with the induction hypothesis establishes the claim. 
\end{proof}

We now introduce several applications of Theorem~\ref{thm_main}. Recall from Lemma~\ref{lem_Y_i_for_simple_poly} that if $P$ admits a conic sequence such that  $\mathscr{C}_{\{\mc{P}_i\}}$ consists of simple polytopes, then each $Y_i$ is an orbifold $S^1$-fiber bundle over a toric orbifold. The following proposition has been studied in~\cite{Luo-Thesis}, where the author mainly discusses smooth cases. However, if we restrict our attention to the singular  cohomology with rational coefficients, then the same statement holds for an orbifold $S^1$-fiber bundle over a toric orbifold. The main tool for the proof is the Gysin sequence. We refer to \cite[Section 5]{PS} for the Gysin sequence in the orbifold setup. 
\begin{proposition}\cite[Remark 2.5.10, Theorem 2.5.48]{Luo-Thesis} \label{prop_cohom_S1-bdl}
Let $C_i$ be a simple polytope and $Y_i=\pi^{-1}(C_i)$ as above. Then $H^k(Y_i)=0$ if $k$ is even  with~$k > \dim C_i$ or if $k$ is odd with~$k\leq \dim C_i $.
\end{proposition}

\begin{theorem}\label{thm_vanishing_odd_greater_than_half}
Let $P$ be a lattice polytope admitting a conic sequence as in \eqref{eq_conic_decomp} such that  $\mathscr{C}_{\{\mc{P}_i\}}$ consists of simple polytopes, and $X_P$  be the associated toric variety.
Then $H^{2k-1}(X)=0$ for all $2k-1> \frac{1}{2}\dim X_P$. 
\end{theorem}
\begin{proof}
Similarly to the proof of Theorem \ref{thm_main}, we prove the claim by the induction on the filtration in \eqref{eq_filtration_on_X}. Let $2d_i \colonequals \dim X_i$ for each $1\leq i \leq \ell$.  Then we have  $d_{i}\geq d_{i-1}$ and the condition for a conic sequence on $P$ implies that 
\begin{equation}\label{eq_dimYi}
    \dim Y_i\leq 2d_{i}-1.
\end{equation} 
We assume that $H^j(X_{i-1})=0$ for odd $j$ with $j > d_{i-1}$, and claim that $H^j(X_i)=0$ for odd $j$ with $j > d_i$. 

Consider the long exact sequence in \eqref{eq_les_from_cofib} induced from the cofibration in \eqref{eq_cofibration}. Since $\dim Y_i =2\dim C_i+1$, Proposition \ref{prop_cohom_S1-bdl} implies that $\widetilde{H}^{j-1}(Y_i)=0$ for all even $j-1$ with $j-1 > \frac{1}{2}(\dim Y_i-1)$. In particular, since 
\[
d_i \geq \dim C_i+1 = \frac{1}{2}(\dim Y_i+1 ),
\]
by \eqref{eq_dimYi}, we have $\widetilde{H}^{j-1}(Y_i)=0$ for all even $j-1$ with $j-1 \geq d_{i}$. Moreover, the induction hypothesis implies that $\widetilde H^j(X_{i-1})=0$ for all odd $j$ with $j > d_{i-1}$. 
Since $d_{i}\geq d_{i-1}$, we have $\widetilde H^j(X_{i-1})=0$ for all odd $j > d_{i}$. Therefore, $H^j(X_{i})=0$ if $j$ is odd with $j > \frac{1}{2}\dim X_i$, which completes the induction. 
\end{proof}

Especially, if a $3$-dimensional polytope $P$ admits a conic sequence, then $\mathscr{C}_{\{\mc{P}_i\}}$ consists of simple polytopes because every element in $\mathscr{C}_{\{\mc{P}_i\}}$ is of dimension at most~$2$. Hence Theorem \ref{thm_vanishing_odd_greater_than_half} implies that $H^5(X_P)=0$ in this case. We refer to  \cite[Proposition 3.5.3-(f)]{Jor}.

The next corollary specializes the $\Delta$-conic sequence of $P$, which is proved in \cite[Corollary 3.7]{SaSo2} as well. 
\begin{corollary}\label{cor_triangle_conic_implies_odd_vanish}
If $P$ is $\Delta$-conic, then $H^{\text{odd}}(X_P)=0$.
\end{corollary}
\begin{proof}

By Lemma \ref{lem_Y_i_simplex}, each $Y_i$ is homeomorphic to $S^{2\dim C_i+1}/K_i$ for some finite group $K_i$. Now the transfer homomorphism (see for instance \cite[III.2]{Bor}) implies that 
\begin{equation}\label{eq_transfer}
\widetilde{H}^j(S^{2\dim C_i+1}/K_i)=\begin{cases} \QQ & \text{ if }j=2\dim C_i+1,\\
0 & \text{ otherwise.}\end{cases}
\end{equation}
Hence the proof follows directly from Theorem \ref{thm_main}. 
\end{proof}

The converse of Corollary \ref{cor_triangle_conic_implies_odd_vanish} is not true in general. Indeed, G. Barthel, et. al. \cite[Example 3.5]{BBFK} introduced two examples of toric varieties $X$ and $X'$ over the polytope $P$ given in Example \ref{ex_triangular_prism} such that $H^3(X)=0$ and $H^3(X')=\QQ$. Note that $H^1(X)=H^1(X')=0$ because a projective toric variety is simply connected \cite[Theorem 12.1.10]{CLS} and $H^5(X)=H^5(X')=0$ by Theorem \ref{thm_vanishing_odd_greater_than_half}. 

Our last application of Theorem \ref{thm_main} is the Poincar\'e polynomial 
\[
\Poin(X_P, t)=\sum_{i=0}^{\dim X_P} \beta_{i}t^i
\]
of a toric variety $X_P$, where $\beta_{i}$ denotes the $i$-th cohomology Betti number. When $P$ is a simple lattice polytope, that is,  the resulting toric variety $X_P$ is a toric orbifold, it is well-known that $H^\ast(X_P)$ vanishes in odd degrees. One can also conclude this using Corollary \ref{cor_triangle_conic_implies_odd_vanish} because every simple polytope admits at least one $\Delta$-conic sequence. In this case, $(\beta_{0}, \dots, \beta_{2i}, \dots, \beta_{2n})$ agrees with the $h$-vector $h(P)=(h_0, \dots, h_i, \dots, h_n)$ of $P$, namely, the vector obtained by the equation given in \eqref{eq_fvec_hvec_rel}.
We refer to \cite{McM, Sta}. 

The next theorem allows us to enjoy Equation \eqref{eq_fvec_hvec_rel} in a wider class of polytopes including all simple polytopes, namely, the category of polytopes admitting $\Delta$-conic sequences.

\begin{theorem}\label{thm_poincare_poly}
Let $X_P$ be the toric variety associated with a lattice polytope with a $\Delta$-conic sequence. Then the Poincar\'e polynomial $\Poin(X_P, t)$ of $X_P$  is given by 
\begin{equation}\label{eq_poincare_poly_using_fvec}
\mathsf{Poin}(X_P, t)= \sum_{k=0}^n f_k (t^2-1)^k. %= \sum_{k=0}^n h_k x^k,
\end{equation} where $f_k$ is the number of $k$-dimensional faces of $P$.
\end{theorem}
\begin{proof}
Let $\mathscr{C}_{\{\mc{P}_i\}}=\{C_2, \dots, C_\ell\} = \{\Delta^{d_2-1}, \dots, \Delta^{d_\ell-1}\}$ be the sequence of simplices corresponding to the given $\Delta$-conic sequence $\{\mc{P}_i\}_{1\leq i \leq \ell}$ of $P$. Notice that $d_2=1$, $d_{\ell}=n$ and $1\leq d_i \leq  n-1$ for $i=2, \dots, \ell-1$. By Lemma \ref{lem_Y_i_simplex} together with \eqref{eq_transfer}, we have that $\widetilde{H}^{2d_i-1}(Y_i)$ is of rank $1$ and $\widetilde{H}^j(Y_i)$ vanishes for $j< 2d_i-1$. Therefore, the long exact sequence in \eqref{eq_les_from_cofib} implies that the rank of $\widetilde{H}^\ast(X_{i})$ increases by $1$ whenever $\ast$ meets $2d_i$. Repeating this process for each $i=2, \dots, \ell$, we conclude that 
\begin{equation}\label{eq_h_k}
\beta_{2j}=\# \{i \mid d_i=j\}.
\end{equation}

Since $C_i=\Delta^{d_i-1}$ for some $d_i\geq 1$, we have 
$\Phi_{C_i}(x)=\sum_{k=1}^{d_i} { d_i \choose  k} x^{k-1}$. Thus we get 
\begin{equation}\label{eq_fvec_cone}
1+x\Phi_{C_i}(x)=(x+1)^{d_i}.
\end{equation}
Therefore, the result of Proposition \ref{prop_f_vec} together with \eqref{eq_h_k} and \eqref{eq_fvec_cone} implies that 
\begin{equation}\label{eq_fvec_bettinum}
\sum_{k=0}^n f_k(P) x^k = \sum_{i=2}^{\ell} (x+1)^{d_i} = \sum_{k=0}^n \beta_{2k}(x+1)^k.
\end{equation}
Since $H^\ast(X_P)$ vanishes in odd degrees, the result follows by plugging $t^2-1$ into $x$ in \eqref{eq_fvec_bettinum}. 
\end{proof}

\begin{figure}
\tdplotsetmaincoords{80}{40}
\begin{tikzpicture}[tdplot_main_coords, scale=0.5, yscale=0.9]

\draw[-stealth, gray] (0,0,0) -- (3,0,0) node[anchor=north west]{\tiny$x_1$};
\draw[-stealth, gray] (0,0,0) -- (0,3,0) node[anchor=south]{\tiny$x_2$};
\draw[-stealth, gray] (0,0,0) -- (0,0,2.5) node[anchor=south]{\tiny$x_3$};
\node[fill=red, red, circle, inner sep=.9pt] at (0,1,0) {};

\begin{scope}[xshift=100]
\draw[-stealth, gray] (0,0,0) -- (3,0,0);% node[anchor=north west]{$x_1$};
\draw[-stealth, gray] (0,0,0) -- (0,3,0);% node[anchor=north west]{$x_2$};
\draw[-stealth, gray] (0,0,0) -- (0,0,2.5);% node[anchor=south]{$x_3$};
\draw[red, thick] (0,1,0)--(0,1,1);
\node[fill=red, red, circle, inner sep=.9pt] at (0,1,1) {};
\node[fill, circle, inner sep=.7pt] at (0,1,0) {};
\end{scope}

\begin{scope}[xshift=200]
\draw[-stealth, gray] (0,0,0) -- (3,0,0);% node[anchor=north west]{$x_1$};
\draw[-stealth, gray] (0,0,0) -- (0,3,0);% node[anchor=north west]{$x_2$};
\draw[-stealth, gray] (0,0,0) -- (0,0,2.5);% node[anchor=south]{$x_3$};

\draw (0,1,0)--(0,1,1);
\draw[red, thick] (0,2,0)--(0,1,0);
\node[fill=red, red, circle, inner sep=.9pt] at (0,2,0) {};
\node[fill, circle, inner sep=.7pt] at (0,1,1) {};
\node[fill, circle, inner sep=.7pt] at (0,1,0) {};
\end{scope}

\begin{scope}[xshift=300]
\draw[-stealth, gray] (0,0,0) -- (3,0,0);% node[anchor=north west]{$x_1$};
\draw[-stealth, gray] (0,0,0) -- (0,3,0);% node[anchor=north west]{$x_2$};
\draw[-stealth, gray] (0,0,0) -- (0,0,2.5);% node[anchor=south]{$x_3$};

\draw[fill=red!20] (0,2,2)--(0,1,1)--(0,1,0)--(0,2,0)--cycle;
\node[fill=red, red, circle, inner sep=.9pt] at (0,2,2) {};
\draw[red, thick] (0,2,2)--(0,1,1);
\draw[red, thick] (0,2,2)--(0,2,0);

\draw (0,1,0)--(0,1,1);
\draw (0,2,0)--(0,1,0);
\node[fill, circle, inner sep=.7pt] at (0,2,0) {};
\node[fill, circle, inner sep=.7pt] at (0,1,1) {};
\node[fill, circle, inner sep=.7pt] at (0,1,0) {};
\end{scope}

\begin{scope}[xshift=400]
\draw[-stealth, gray] (0,0,0) -- (3,0,0);% node[anchor=north west]{$x_1$};
\draw[-stealth, gray] (0,0,0) -- (0,3,0);% node[anchor=north west]{$x_2$};
\draw[-stealth, gray] (0,0,0) -- (0,0,2.5);% node[anchor=south]{$x_3$};

\draw[fill=blue!20] (0,2,2)--(0,1,1)--(0,1,0)--(0,2,0)--cycle;
\node[fill, circle, inner sep=.7pt] at (0,2,2) {};
\draw (0,2,2)--(0,1,1);
\draw (0,2,2)--(0,2,0);
\draw (0,1,0)--(0,1,1);
\draw (0,2,0)--(0,1,0);

\draw[fill=pink, opacity=0.8] (1,1,1)--(0,1,1)--(0,1,0)--cycle;
\node[fill=red, red, circle, inner sep=.9pt] at (1,1,1) {};
\draw[red, thick] (1,1,1)--(0,1,1);
\draw[red, thick] (1,1,1)--(0,1,0);

\node[fill, circle, inner sep=.7pt] at (0,2,0) {};
\node[fill, circle, inner sep=.7pt] at (0,1,1) {};
\node[fill, circle, inner sep=.7pt] at (0,1,0) {};
\end{scope}

\begin{scope}[xshift=500]
\draw[-stealth, gray] (0,0,0) -- (3,0,0);% node[anchor=north west]{$x_1$};
\draw[-stealth, gray] (0,0,0) -- (0,3,0);% node[anchor=north west]{$x_2$};
\draw[-stealth, gray] (0,0,0) -- (0,0,2.5);% node[anchor=south]{$x_3$};

\draw[fill=blue!20] (0,2,2)--(0,1,1)--(0,1,0)--(0,2,0)--cycle;
\node[fill, circle, inner sep=.7pt] at (0,2,2) {};
\draw (0,2,2)--(0,1,1);
\draw (0,2,2)--(0,2,0);
\draw (0,1,0)--(0,1,1);
\draw (0,2,0)--(0,1,0);

\draw[fill=blue!40, opacity=0.8] (1,1,1)--(0,1,1)--(0,1,0)--cycle;
\draw (1,1,1)--(0,1,1);
\draw (1,1,1)--(0,1,0);

\draw[fill=pink, opacity=0.8] (1,2,1)--(1,1,1)--(0,1,0)--(0,2,0)--cycle;
\node[fill=red, red, circle, inner sep=.9pt] at (1,2,1) {};
\draw[red, thick] (1,2,1)--(1,1,1);
\draw[red, thick] (1,2,1)--(0,2,0);

\node[fill, circle, inner sep=.7pt] at (0,2,0) {};
\node[fill, circle, inner sep=.7pt] at (0,1,1) {};
\node[fill, circle, inner sep=.7pt] at (0,1,0) {};
\node[fill, circle, inner sep=.7pt] at (1,1,1) {};
\end{scope}

\begin{scope}[xshift=600]
\draw[-stealth, gray] (0,0,0) -- (3,0,0);% node[anchor=north west]{$x_1$};
\draw[-stealth, gray] (0,0,0) -- (0,3,0);% node[anchor=north west]{$x_2$};
\draw[-stealth, gray] (0,0,0) -- (0,0,2.5);% node[anchor=south]{$x_3$};

\draw[fill=blue!20] (0,2,2)--(0,1,1)--(0,1,0)--(0,2,0)--cycle;
\draw (0,2,2)--(0,1,1);
\draw (0,2,2)--(0,2,0);
\draw (0,1,0)--(0,1,1);
\draw (0,2,0)--(0,1,0);

\draw[fill=pink, opacity=0.8] (1,2,2)--(0,2,2)--(0,1,1)--(1,1,1)--(1,2,1)--cycle;
\node[fill=red, red, circle, inner sep=.9pt] at (1,2,2) {};
\draw[red, thick] (1,2,2)--(0,2,2);
\draw[red, thick] (1,2,2)--(1,1,1);
\draw[red, thick] (1,2,2)--(1,2,1);

\draw[fill=blue!40, opacity=0.8] (1,1,1)--(0,1,1)--(0,1,0)--cycle;
\draw (1,1,1)--(0,1,1);
\draw (1,1,1)--(0,1,0);

\draw[fill=blue!40, opacity=0.8] (1,2,1)--(1,1,1)--(0,1,0)--(0,2,0)--cycle;
\draw (1,2,1)--(1,1,1);
\draw (1,2,1)--(0,2,0);

\node[fill, circle, inner sep=.7pt] at (0,2,2) {};
\node[fill, circle, inner sep=.7pt] at (0,2,0) {};
\node[fill, circle, inner sep=.7pt] at (0,1,1) {};
\node[fill, circle, inner sep=.7pt] at (0,1,0) {};
\node[fill, circle, inner sep=.7pt] at (1,1,1) {};
\node[fill, circle, inner sep=.7pt] at (1,2,1) {};
\end{scope}

%%%%	Second Conic of GZ		%%%%%%%%%%%%%%

\begin{scope}[yshift=-120]

\draw[-stealth, gray] (0,0,0) -- (3,0,0) node[anchor=north west]{\tiny$x_1$};
\draw[-stealth, gray] (0,0,0) -- (0,3,0) node[anchor=south]{\tiny$x_2$};
\draw[-stealth, gray] (0,0,0) -- (0,0,2.5) node[anchor=south]{\tiny$x_3$};
\node[fill=red, red, circle, inner sep=.9pt] at (0,1,0) {};

\begin{scope}[xshift=100]
\draw[-stealth, gray] (0,0,0) -- (3,0,0);% node[anchor=north west]{$x_1$};
\draw[-stealth, gray] (0,0,0) -- (0,3,0);% node[anchor=north west]{$x_2$};
\draw[-stealth, gray] (0,0,0) -- (0,0,2.5);% node[anchor=south]{$x_3$};
\draw[red, thick] (0,1,0)--(0,1,1);
\node[fill=red, red, circle, inner sep=.9pt] at (0,1,1) {};
\node[fill, circle, inner sep=.7pt] at (0,1,0) {};
\end{scope}

\begin{scope}[xshift=200]
\draw[-stealth, gray] (0,0,0) -- (3,0,0);% node[anchor=north west]{$x_1$};
\draw[-stealth, gray] (0,0,0) -- (0,3,0);% node[anchor=north west]{$x_2$};
\draw[-stealth, gray] (0,0,0) -- (0,0,2.5);% node[anchor=south]{$x_3$};

\draw (0,1,0)--(0,1,1);
\draw[red, thick] (0,2,0)--(0,1,0);
\node[fill=red, red, circle, inner sep=.9pt] at (0,2,0) {};
\node[fill, circle, inner sep=.7pt] at (0,1,1) {};
\node[fill, circle, inner sep=.7pt] at (0,1,0) {};
\end{scope}

\begin{scope}[xshift=300]
\draw[-stealth, gray] (0,0,0) -- (3,0,0);% node[anchor=north west]{$x_1$};
\draw[-stealth, gray] (0,0,0) -- (0,3,0);% node[anchor=north west]{$x_2$};
\draw[-stealth, gray] (0,0,0) -- (0,0,2.5);% node[anchor=south]{$x_3$};

\draw[fill=red!20] (0,2,2)--(0,1,1)--(0,1,0)--(0,2,0)--cycle;
\node[fill=red, red, circle, inner sep=.9pt] at (0,2,2) {};
\draw[red, thick] (0,2,2)--(0,1,1);
\draw[red, thick] (0,2,2)--(0,2,0);

\draw (0,1,0)--(0,1,1);
\draw (0,2,0)--(0,1,0);
\node[fill, circle, inner sep=.7pt] at (0,2,0) {};
\node[fill, circle, inner sep=.7pt] at (0,1,1) {};
\node[fill, circle, inner sep=.7pt] at (0,1,0) {};
\end{scope}

\begin{scope}[xshift=400]
\draw[-stealth, gray] (0,0,0) -- (3,0,0);% node[anchor=north west]{$x_1$};
\draw[-stealth, gray] (0,0,0) -- (0,3,0);% node[anchor=north west]{$x_2$};
\draw[-stealth, gray] (0,0,0) -- (0,0,2.5);% node[anchor=south]{$x_3$};

\draw[fill=blue!20] (0,2,2)--(0,1,1)--(0,1,0)--(0,2,0)--cycle;
\node[fill, circle, inner sep=.7pt] at (0,2,2) {};
\draw (0,2,2)--(0,1,1);
\draw (0,2,2)--(0,2,0);
\draw (0,1,0)--(0,1,1);
\draw (0,2,0)--(0,1,0);

\node[fill=red, red, circle, inner sep=.9pt] at (1,2,1) {};
\draw[red, thick] (1,2,1)--(0,2,0);

\node[fill, circle, inner sep=.7pt] at (0,2,0) {};
\node[fill, circle, inner sep=.7pt] at (0,1,1) {};
\node[fill, circle, inner sep=.7pt] at (0,1,0) {};
\end{scope}

\begin{scope}[xshift=500]
\draw[-stealth, gray] (0,0,0) -- (3,0,0);% node[anchor=north west]{$x_1$};
\draw[-stealth, gray] (0,0,0) -- (0,3,0);% node[anchor=north west]{$x_2$};
\draw[-stealth, gray] (0,0,0) -- (0,0,2.5);% node[anchor=south]{$x_3$};

\draw[fill=blue!20] (0,2,2)--(0,1,1)--(0,1,0)--(0,2,0)--cycle;
\draw (0,2,2)--(0,1,1);
\draw (0,2,2)--(0,2,0);
\draw (0,1,0)--(0,1,1);
\draw (0,2,0)--(0,1,0);
\draw (1,2,1)--(0,2,0);

\draw[fill=red!20] (1,2,2)--(0,2,2)--(0,2,0)--(1,2,1)--cycle;
\node[fill=red, red, circle, inner sep=.9pt] at (1,2,2) {};
\draw[red, thick] (1,2,2)--(0,2,2);
\draw[red, thick] (1,2,2)--(1,2,1);

\node[fill, circle, inner sep=.7pt] at (0,2,2) {};
\node[fill, circle, inner sep=.7pt] at (1,2,1) {};
\node[fill, circle, inner sep=.7pt] at (0,2,0) {};
\node[fill, circle, inner sep=.7pt] at (0,1,1) {};
\node[fill, circle, inner sep=.7pt] at (0,1,0) {};
\end{scope}

\begin{scope}[xshift=600]
\draw[-stealth, gray] (0,0,0) -- (3,0,0);% node[anchor=north west]{$x_1$};
\draw[-stealth, gray] (0,0,0) -- (0,3,0);% node[anchor=north west]{$x_2$};
\draw[-stealth, gray] (0,0,0) -- (0,0,2.5);% node[anchor=south]{$x_3$};

\draw[fill=blue!20] (0,2,2)--(0,1,1)--(0,1,0)--(0,2,0)--cycle;
\draw (0,2,2)--(0,1,1);
\draw (0,2,2)--(0,2,0);
\draw (0,1,0)--(0,1,1);
\draw (0,2,0)--(0,1,0);
\draw (1,2,1)--(0,2,0);
\draw (1,2,2)--(0,2,2);
\draw (1,2,2)--(1,2,1);

\draw[fill=blue!20] (1,2,2)--(0,2,2)--(0,2,0)--(1,2,1)--cycle;

\draw[fill=pink, opacity=0.8] (1,2,2)--(0,2,2)--(0,1,1)--(0,1,0)--(0,2,0)--(1,2,1)--cycle;
\node[fill=red, red, circle, inner sep=.9pt] at (1,1,1) {};
\draw[red, thick] (1,1,1)--(1,2,2);
\draw[red, thick] (1,1,1)--(1,2,1);
\draw[red, thick] (1,1,1)--(0,1,0);
\draw[red, thick] (1,1,1)--(0,1,1);

\node[fill, circle, inner sep=.7pt] at (1,2,2) {};
\node[fill, circle, inner sep=.7pt] at (0,2,2) {};
\node[fill, circle, inner sep=.7pt] at (1,2,1) {};
\node[fill, circle, inner sep=.7pt] at (0,2,0) {};
\node[fill, circle, inner sep=.7pt] at (0,1,1) {};
\node[fill, circle, inner sep=.7pt] at (0,1,0) {};
\end{scope}
\end{scope}

\end{tikzpicture}
\caption{Two conic sequences of $GZ_3$.}
\label{fig_conic_GZ}
\end{figure}
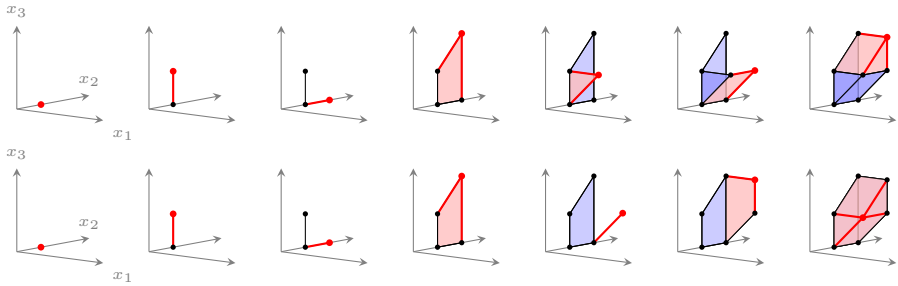

\begin{example}
Let $GZ_3$ be the three-dimensional Gelfand--Zetlin polytope, namely the set of points $(x_1, x_2, x_3)$ in $\RR^3$ satisfying the following inequalities:
\[
\arraycolsep=1.5pt
\begin{array}{ccccccccc}
0&&&&1&&&&2 \\
&\leqmrot&&\leqprot&&\leqmrot&&\leqprot&\\
&&x_1&&&&x_2&&\\
&&&\leqmrot&&\leqprot&&&\\
&&&&x_3&&&&
\end{array}
\]
We describe two conic sequences of $GZ_3$ in Figure \ref{fig_conic_GZ}, where the first one is a $\Delta$-conic sequence. Since $f(GZ_3)=(7, 11, 6, 1)$, Theorem \ref{thm_poincare_poly} implies that the Poincar\'e polynomial of the toric variety $X_{GZ_3}$ is 
\[
\Poin (X_{GZ_3};t)=7+11(t^2-1)+6(t^2-1)^2+(t^2-1)^3=t^3+3t^4+2t^2+1.
\]
\end{example}

\end{document}